\documentclass[a4paper,12pt]{article}
\usepackage{hyperref}
\usepackage{amsmath}

\usepackage{enumerate}

\usepackage{amsthm}
\usepackage{amsfonts}
\usepackage{subeqnarray}
\usepackage{mathrsfs}
\usepackage{latexsym}
\usepackage{dsfont}
\usepackage{mathtools} 

\usepackage[margin=2cm]{geometry}
\allowdisplaybreaks[4]
\usepackage{showkeys}
\newcommand{\ud}{\mathrm{d}}
\newcommand{\beq}{\begin{equation}}
\newcommand{\eeq}{\end{equation}}
\newcommand{\ba}{\begin{eqnarray}}
\newcommand{\ea}{\end{eqnarray}}

\newcommand{\f}{\varphi}
\newcommand{\ms}{\mathscr}

\newtheorem{theorem}{Theorem}[section]
\newtheorem{remark}{Remark}[section]

\newtheorem{proposition}{Proposition}[section]
\newtheorem{lemma}{Lemma}[section]

\title{On the Doi-Edwards and   K-BKZ rheological models for polymer fluids: an existence result  for shear flows.
}

\author{Ionel Sorin Ciuperca$^1$, Arnaud Heibig$^2$ and Liviu Iulian Palade$^2$ \thanks{Corresponding author.  E-mail: ciuperca@math.univ-lyon.fr, arnaud.heibig@insa-lyon.fr, liviu-iulian.palade@insa-lyon.fr; Fax: +33 472438529}   
}

\numberwithin{equation}{section}
\begin{document}

\maketitle

\begin{flushleft}

Universit\'e de Lyon, CNRS, Institut Camille Jordan UMR 5208\\

$^1$ Universit\'e Lyon 1,  B\^at Braconnier, 43 Boulevard du 11 Novembre 1918, F-69622, Villeurbanne, France. 

$^2$ INSA-Lyon,  P\^ole de Math\'ematiques, B\^at. Leonard de Vinci No. 401, 21 Avenue Jean Capelle, F-69621, Villeurbanne, France.

\end{flushleft}

\begin{abstract}
This paper establishes the existence of smooth solutions for the Doi-Edwards rheological model of viscoelastic polymer fluids in shear flows.  The problem  turns out to be formally equivalent to a K-BKZ equation but with constitutive  functions spanning beyond the usual mathematical framework.  We prove, for small enough initial data, that the solution remains in the domain of hyperbolicity of the equation  for all $t\geq0$.  

\end{abstract}

\begin{flushleft}


\end{flushleft}


\section{Introduction.}\label{intro}

Today's modeling  of non-Newtonian and viscoelastic industrial flows (and of the rheological behavior in general) relies heavily on molecular theories.   The rheology of various linear/branched polymer liquids is very well described by the so-called tube-reptation theories initiated by Doi and Edwards (DE), see \cite{de1}.  At the heartcore of any kinetical model one finds  a configurational probability diffusion equation (a parabolic PDE) the solution of which is needed to obtain the stress tensor, i.e. the corresponding constitutive equation (CE).  For the full, non-linear DE model, in \cite{chp1} we proved the existence and uniqueness of solutions for the diffusion equation using the Schauder fixed point theorem and the Galerkin's  approximation method.  Moreover, this work is related to that in \cite{bcip}. 

Here we focus on an equally crucial issue, that of existence of solutions to shear flows.  The corresponding constitutive equation is that of the simplified DE theory commonly called Independent Alignment Approximation (IAA).  The governing equations for the shear flow are given below:

\beq\label{0i1}
\dfrac{\partial v}{\partial t}=\dfrac{\partial \theta}{\partial x}
\eeq 

\beq\label{0i2}
\theta=\int_0^1 \int_{S_2} u_1 u_2 F \ud u \ud s
\eeq

\beq\label{0i3}
\dfrac{\partial F}{\partial t}=\dfrac{\partial^2 F}{\partial s^2}- \dfrac{\partial v}{\partial x}\dfrac{\partial}{\partial u}\cdot \left(\mathcal{G}_0(u)F \right) 
\eeq

In the above, the notations are common to the mathematical and the related continuum mechanics, rheology,  and polymer physics literature:  $v=v(x,t)$ is the scalar velocity field, $\theta=\theta(x,t)$ is the stress, and $F(t,u,s,x)$ the configurational probability function.  The flow occurs in the $x$ direction during time $t$, $s\in(0,1)$ is the polymer chain's primitive path curvilinear coordinate, and $u=(u_1,u_2,u_3)$ the unitary vector pointing outwardly the unit sphere $S_2$.  Similar to notations in \cite{de0}, $\mathcal{G}_0(u)= M\cdot u-(M:uu)u$, with

\begin{eqnarray}\label{0im}
M=
\left( \begin{array}{ccc}
0 & 1 & 0 \\
0 & 0 & 0 \\
0 & 0 & 0 
\end{array} \right)
\end{eqnarray}
     
To the system of equations \eqref{0i1}-\eqref{0i3} we assign the following boundary and initial conditions:

\begin{equation}\label{0ibc}
\begin{cases}
v=0, \,\text{for}\, x\in\partial\Omega \\
v=v_0, \,\text{for}\, t=0 \\
F=\dfrac{1}{4\pi}, \,\text{for}\, s=0 \,\text{or}\,s=1\\
F=F_0, \,\text{for}\, t=0  
\end{cases}
\end{equation}

where $\Omega\subset\mathbb{R}$ is the range for $x$, while $v_0(x)$ and $F_0(u,s,x)$ are initial data.  

From \cite{de0} one sees the equation \eqref{0i3} for $F$ can be solved allowing the obtainment of $\theta$ as a function of the velocity gradient $\dfrac{\partial v}{\partial x}$.  In particular, for $F_0=1/(4\pi)$ one gets:

\begin{align}\label{0i4}
& \theta = - g_{DE} \left(\int_0^t \dfrac{\partial v}{\partial x}(x,\tau)\ud\tau  \right)a_{DE}(t)+\int_0^t g_{DE} \left(\int_\tau^t \dfrac{\partial v}{\partial x}(x,r)\ud r \right)a'_{DE}(t-\tau)\ud\tau
\end{align}

with $g_{DE}:\mathbb{R}\to \mathbb{R} $ 

\begin{align}\label{0i5}
& g_{DE}(y)=-\int_{S_2} \dfrac{u_1 u_2}{\left[(u_1-u_2 y)^2+u_2^2+u_3^2\right]^{3/2}} \ud u,\, \forall y\in \mathbb{R}
\end{align}

and $a_{DE}:\mathbb{R}_+\to \mathbb{R} $

\begin{align}\label{0i6}
& a_{DE}(t)=\sum_{p=1}^{+\infty} \dfrac{1}{(2p+1)^2}e^{-(2p+1)^2t}
\end{align}

the relaxation function of the DE model.

From the above considerations one infers the shear flow problem under scrutiny is tantamount to solving for $v$ the below integro-differential equation:

\begin{align}\label{0i7}
& \dfrac{\partial v}{\partial t}= - \dfrac{\partial }{\partial x} g_{DE} \left(\int_0^t \dfrac{\partial v}{\partial x}(x,\tau)\ud\tau  \right)a_{DE}(t)+\dfrac{\partial }{\partial x} \int_0^t g_{DE} \left(\int_\tau^t \dfrac{\partial v}{\partial x}(x,r)\ud r \right)a'_{DE}(t-\tau)\ud\tau,\, t>0 
\end{align}

Now equation \eqref{0i7} - here obtained on molecular dynamics grounds - has been focused on  within the area of viscoelastic fluids as it comes out when one studies shear flows for the K-BKZ fluids.  There is no contingency here as in their 1978 original paper \cite{de0}, Doi and Edwards have shown the simplified IAA version of their nonlinear model actually enters the class of K-BKZ integral models, which are based on continuum mechanics concepts (for more on see  \cite{bird2}, \cite{ott1}, \cite{ren}).  Consequently, when undertaking the study of certain particular flows of DE fluids one may capitalize on previously obtained results for K-BKZ liquids.  

In this paper we study equation \eqref{0i7} with more general functions $g$ and $a$ replacing $g_{DE}$ and $a_{DE}$, respectively.  We prove a global in time solution existence result for small enough data.  Uniqueness is the focus of an upcoming paper \cite{hbg}.   Equation \eqref{0i7} - as well as variants of it - was studied by various authors, see Renardy, Hrusa and Nohel \cite{rhn}, Engler \cite{enl1}, Brandon and Hrusa \cite{bh1} and references cited therein.

The existence of local in time solutions \cite{rhn} and of global solutions \cite{enl1}, \cite{bh1} are known under more restrictive conditions compared to those stated in this paper.  One of the assumptions in \cite{enl1} and \cite{bh1} is $g'(y)<-\gamma$, for any $y\in\mathbb{R}$, with $\gamma>0$, which is not verified by the function $g=g_{DE}$.  Here we make use of the less restrictive assumption $g'(y)<0$, for any $y\in[-\theta,\theta]$, with $\theta>0$, and show that the argument of $g'$ is confined to $[-\theta,\theta]$. The requirement $g'<0$ is a necessary hyperbolicity condition for the solution local existence.   For the work presented in this paper, this condition being valid only locally makes it necessary to control, w.r.t. time $t$, the argument $\displaystyle \int_0^t \dfrac{\partial v}{\partial x}(x,\tau) \ud \tau $ of $g'$.  Observe that at a first sight, this argument may become large with increasing $t$.    

Next, among the restrictive hypotheses invoked by the authors of \cite{bh1} for function $a$ is that $a''\in L^1(0,+\infty)$, which  $a=a_{DE}$ does not verify.  Comparatively, here we shall place significantly less restrictions  on $a$ and accordingly will construct a class of totally monotone functions, an element of which is $a=a_{DE}$.


The manuscript is organized as following:

In Section \ref{ibvps} we introduce the problem and enunciate the main result.  

Section \ref{prl} is devoted to the proof of several necessary results such as a G\aa rding type inequality and an inversion formula  for the operator $u\mapsto a*u$ which differs from the one given in \cite{bh1}.  

In Section \ref{appe} we introduce an approximated problem and obtain useful estimates for its solution.  In particular we obtain an estimate for the argument of $g'$ with the help of a maximal function.  The proof of the main result is achieved in Section \ref{pmrs}.

In the ending Section \ref{ax} we construct a class of totally monotone functions that is compatible with the hypothesis made about $a$.

\section{Presentation of the problem and of the main results.}\label{ibvps}

Let from now on $\Omega\subset\mathbb{R}$ be a bounded, open interval.  Let the functions $f:\Omega\times [0,+\infty)\to\mathbb{R}$, $g:I\subset\mathbb{R}\to\mathbb{R}$, with $I\ni0$ an open interval,  $v_0:\Omega\to\mathbb{R}$, $a:[0,+\infty)\to\mathbb{R}$.   

The aim is to search for  a solution  $v:\Omega\times[0,+\infty)\to\mathbb{R}$ to the below given initial boundary value problem:

\begin{align}
& v_t(x,t)=-a(t)\dfrac{\partial}{\partial x} g\left( \int_0^t v_x(x,s)\ud s \right) + \dfrac{\partial}{\partial x}\int_0^t g\left( \int_s^t v_x(x,\tau)\ud \tau \right)a'(t-s)\ud s+f(x,t) \label{p1} \\
& v(x,t=0)=v_0(x),\, \forall x\in\Omega,\, \text{and}\,v(x,t)=0,\,\forall t<0 \label{p2}\\
& v=0, \, \forall x\in \partial\Omega, \forall t\geq0 \label{p3}
\end{align}

In the above, $v_x\equiv \dfrac{\partial v}{\partial x}$ and $a'$ stands for the derivative of $a$.  Throughout this paper, any function defined for $t\geq0$ is  understood as being set equal to $0$ for $t<0$, i.e. it has domain $\mathbb{R}$.   Moreover, for a function $\f\in W^{k,1}(0,+\infty)$ we denote by $\f^{(k)}$ the distributional derivative of $\f$ on $\mathbb{R}^*_+$, derivative which is understood to be extended to $\mathbb{R}$ by 0.  Define

$$\bar{v}^t(x,s):=\int_{t-s}^t v(x,\tau)\ud\tau,\,0\leq s,t;\, x\in\Omega$$ 

Equation \eqref{p1} now takes on a simpler form:

\beq\label{p4}
v_t(x,t)=\int_0^{+\infty}a'(s) \dfrac{\partial}{\partial x}g\left(\bar{v}^t_x(x,s) \right) \ud s +f(x,t)
\eeq

Drawing inspiration from \cite{bh1}, \eqref{p4} can be re-written as

\beq\label{p5}
v_t(x,t)+g'(0)\int_0^t a(t-s) v_{xx}(x,s)\ud s=f(x,t)+\mathcal{G}(x,t)
\eeq

where

\begin{align}\label{p6}
\mathcal{G}(x,t) & =\int_0^{+\infty}a'(s)\left[g'\left(\bar{v}^t_{x}(x,s) \right) -g'(0) \right]\bar{v}^t_{xx}(x,s) \ud s \nonumber\\
& = \int_0^t v_{xx}(x,s)\int_{t-s}^{+\infty}a'(\tau)\left[ g' \left(\bar{v}^t_{x}(x,\tau) \right) -g'(0) \right] \ud\tau \ud s
\end{align}

Convolution with respect to $t$ is denoted as usually by $*$; therefore \eqref{p5} can be re-written in a more close form as 

$$v_t+g'(0)a*v_{xx}=f+\mathcal{G} $$

We now proceed to presenting several constitutive assumptions.  The function $g$ is taken such that:

\begin{enumerate}[$(g_1).$]
\item there exist $\theta\in[0,1]$ and $K>0$, such that $g\in \ms{C}^3\left([-\theta,\theta],\mathbb{R} \right) $ and $\left|g^{(3)}(y)-g^{(3)}(0) \right|\leq K|y|,\, \forall y\in[-\theta,\theta]$ 
\item $g(0)=g''(0)=0$
\item $g'(0)<0$
\end{enumerate}

The function $f$ is such that 

\begin{enumerate}[$(f_1).$]
\item $f,f_x,f_t\in \ms{C}^0_b\left([0,+\infty);L^2(\Omega)\right)\cap L^2\left([0,+\infty);L^2(\Omega) \right) $, 
\item $f_{tt}\in L^2\left([0,+\infty);L^2(\Omega) \right) $, $\displaystyle\int_0^t f(x,s)\ud s \in \ms{C}^0_b\left([0,+\infty);H^1(\Omega)\right)$, 
\end{enumerate}

where $\ms{C}^0_b\left([0,+\infty);X\right)$ is the set of all functions $w: [0,+\infty)\to X$ which are bounded and continous, and $X$ is a Banach space.

Next, let $v_0$ be such that

\begin{enumerate}[$(v_0)_1.$]
\item $v_0\in H^2(\Omega)$.
\end{enumerate}

We assume that $f$ and $v_0$ are compatible with the already stated initial-boundary conditions: 

\beq\label{nbce1}
v_0(x)=f(x,t=0)=0, \,\forall x\in \partial\Omega
\eeq  

Let the measures associated to $f$ and $v_0$ be defined as:

\begin{align}\label{msf}
F(f):= & \displaystyle\sup_{t\geq0} \int_\Omega \left[ f^2+f_x^2+f_t^2 
+ \left(\int_0^t f(x,s)\ud s \right)^2 + \left(\int_0^t f_x(x,s)\ud s \right)^2
\right] \ud x
\\
& +\int_0^{+\infty}\int_\Omega \left(f^2+f_x^2+f^2_t+f^2_{tt} \right)(x,t)\ud x \ud t 
\end{align}

\beq\label{msv}
V_0(v_0)=\|v_0\|^2_{H^2(\Omega)}=\int_\Omega \left[v_0^2+(v'_0)^2+(v''_0)^2 \right](x)\ud x 
\eeq

For any function $\f\in L^1\left((0,+\infty) \right)$ we denote by $\mathcal{F} \f$ (or alternatively by $\hat{\f}$) and  $\mathcal{L} \f$ the corresponding Fourier and Laplace transforms, i.e.:

$$ \mathcal{F}\f(\omega) := \int_0^{+\infty}\f(t) e^{-i\omega t}\ud t,\, \forall \omega\in \mathbb{R}$$ 

$$ \mathcal{L}\f(z) := \int_0^{+\infty}\f(t) e^{-z t}\ud t, \, \forall z\in \mathbb{C}, \text{Re}{z}\geq0$$

Let us now assume the function $a$ is such that 

\begin{enumerate}[$(\text{a}_1).$]
\item $a\in W^{1,1}\left( 0,+\infty \right) $, $a'(t) \leq 0$ \ a.e. \ $t \geq 0$, \\*[2ex]
There exists a sequence of functions $\left( a_n\right) _{n\in\mathbb{N}}$, $a_n\in \mathscr{C}^2
\left( [0,+\infty)\cap W^{2,\infty}([0,+\infty) \right)$ s.t.
\item  $a_n'(t) \leq 0 \; \; 
\forall \, t \geq 0$,
such that $\left( a_n\right)_{n\in\mathbb{N}} $ bounded in $W^{1,1}\left( 0,+\infty \right)$ and 
$\displaystyle a_n\xrightarrow[n\to+\infty]{\mathscr{D}'\left( 0,+\infty\right)}a  $,
\item $\displaystyle\mathop{\sup}_{n\in\mathbb{N}} \left[ \int_0^1 t\left|a''_n(t)\right|\ud t+\int_1^{+\infty}\sqrt{t}\left|a''_n(t)\right|\ud t+\int_1^{+\infty} t^2\left|a'_n(t)\right|\ud t  \right]<+\infty $,
\item there exist constants $M_1>0$ and  $n_0 \in \mathbb{N}$ s.t. $\text{Re}\left(\mathcal{F}a_n(\omega) \right)\geq \dfrac{M_1}{1+\omega^2}$, $\forall n\in\mathbb{N}$, $n\geq n_0$, $\forall \omega\in\mathbb{R}$; observe that this is a strong positivity condition, common for this type of problems (see \cite{bh1}).
\item  there exist constants $M_2>0$ and $p\in\mathbb{N}^*$ s.t. $\displaystyle 
\dfrac{\left[ \mathcal{F}\left( a'_n\right)\right] ^p  }{\mathcal{F} a_n}\in \mathcal{F}
\left(B_{L^1(\mathbb{R})}(0,M_2) \right) $, $\forall n\in \mathbb{N}$, where 
$B_{L^1(\mathbb{R})}(0,M_2)$ denotes the ball in $L^1(\mathbb{R})$ centered at $0$ and of radius 
$M_2$;  this assumption will be used to obtain a representation for the solution 
$u$ of $a_n * u=b$ (see Theorem \ref{il}).  
\end{enumerate}

\begin{remark}
In Section \ref{ax} we shall construct a class of functions compliant with assumptions $(a_1)$ to $(a_5)$.  This class contains the Doi-Edwards relaxation kernel $ a_{DE}:[0,+\infty)\to\mathbb{R}$,

\beq\label{tila}
a_{DE}(t)=\displaystyle \mathop{\sum}_{k\geq1}\dfrac{1}{(2k+1)^2}e^{-(2k+1)^2 t}
\eeq

Also, since $g_{DE}\in \mathscr{C}^\infty\left(\mathbb{R} \right) $ is an odd function and $\displaystyle g'_{DE}(0)=-3\displaystyle \int_{S_2} u_1^2 u_2^2 \ud u <0$, then  $\displaystyle g_{DE}$ also verifies $(g_1)$-$(g_4)$ and this paper results equally apply to the function $g_{DE}$.  

\end{remark}

The main result of this paper is stated below:

\begin{theorem}[{\bf Main Result}]\label{mr}
Assume that the hypotheses on the data given in  $(g_1)$-$(g_4)$, $(f_1)$-$(f_2)$, $(v_0)_1$, $(a_1)$-$(a_5)$ and \eqref{nbce1} hold true. Then there exists a $\delta>0$ such that, if the additional smallness assumption $F(f)+V_0(v_0)\leq \delta$ is verified, then there exists at least a solution 

$$v\in \displaystyle \left\lbrace \bigcap_{m=0}^2 W^{m,\infty} \left((0,+\infty);H^{2-m}\left(\Omega\right) \right)\right\rbrace  \cap \left\lbrace \bigcap_{m=0}^2 W^{m,2} \left((0,+\infty);H^{2-m}\left(\Omega\right) \right)  \right\rbrace $$ 

with 

$$\displaystyle \int_0^t v(x,s)\ud s \in L^\infty \left((0,+\infty);H^{3}\left(\Omega\right) \right) $$ 

to the problem \eqref{p4}, \eqref{p2}-\eqref{p3}.

\end{theorem}

Next we take on to introducing (and explaining) the proof stages for the aforementioned  Theorem \ref{mr}.   In short, first we obtain a regularized problem $(P_n)$ obtained from \eqref{p5} with $a$ being replaced by a sequence $a_n$ satisfying hypotheses $(a_1)$ to $(a_4)$.  Doing this allows to obtain a local in time existence and uniqueness result capitalizing on Renardy's result in \cite{ren}.   Next goal is to obtain estimates independent of $n$  granting the global existence of the solution for the approximated problem $(P_n)$ and in the end, letting $n\to+\infty$, obtaining our result.  How to get these estimates is explained below.  

Let $u(x,t)=\displaystyle\int_0^t v(x,\tau)\ud\tau$.  For any $t>0$, let $\mathcal{E}(t)$ stand for the sum of squared $L^\infty_t L^2_x$ norms of all derivatives in $x$ and $t$ of $u$ up to third order and of all squared $L^2_t L^2_x$ norms of all derivatives in $x$ and $t$ of $v$ up to second order (see \eqref{pr2}).  We prove that if $\mathcal{E}(t)$ is ``small'' for $t$ close to $0$ (a consequence of the assumption made on data $v_0$ and $f$), then $\mathcal{E}(t)$ stays ``small'' for any $t$.  We do this by obtaining an inequality of the type 

\beq\label{expl1}
\mathcal{E}(t)\leq \dfrac{1}{2}\mathcal{E}(t)+\text{``small enough'' quantities depending uniquely on}\, V_0\,\text{and}\, F
\eeq

Getting the second term in the rhs of  \eqref{expl1} requires previously calculated upper bounds of $v$ and its up to second order derivatives in $x$ and $t$, and of $u$ and its up to third order derivatives in $x$ and $t$.  Equation \eqref{p5} is equivalently written as:

\beq\label{expl2}
v_t+g'(0) a*v_{xx}=f+\mathcal{G}
\eeq

Next, we calculate three energy estimates (in a way similar in nature with that of Brandon and Hrusa \cite{bh1}: we derivate \eqref{expl2} $i$-times (with  $i\in\left\lbrace 0,1,2\right\rbrace $) w.r.t. time $t$ , then multiply the result by $\dfrac{d^i v}{dt^i}$ and integrate on $Q_t:=\Omega\times (0,t)$.  To calculate the second order derivative one uses a finite difference operator $\triangle_h w(t)=w(t+h)-w(t)$, see \eqref{fd1}.  We sum up the resulting three equations and get an equality in which the most important term originates from the convolution part in the \textit{lhs} of \eqref{expl2}.  This term reads

\beq\label{expl3}
g'(0) \left[Q\left( v_x,t,a\right) +Q\left( v_{xt},t,a\right) +Q\left(v_{xtt},t,a \right)  \right] 
\eeq

where $Q(w,t,a)=\displaystyle\int_0^t\int_\Omega w(x,s) \left(a*w \right)(x,s)\ud x \ud s$ (see \eqref{qdf}).  We lower bound \eqref{expl3} using the Plancherel-Parseval equality and assumption $(a_4)$ and get (with $w=0$ outside $(0,t)$)

\beq\label{expl4}
Q(w,t,a)\geq \int_\mathbb{R} \int_\Omega \dfrac{M_1}{1+\omega^2}\left|\left( \mathcal{F}w\right) (x,\omega) \right|^2\ud x \ud \omega
\eeq

Notice the presence of $\dfrac{M_1}{1+\omega^2}$ does not render the \textit{rhs} of \eqref{expl4} sufficiently coercive, however we use it to obtain the necessary coercivity for $Q(w,t,a)+Q\left(w_t,t,a \right) $ instead of $Q(w,t,a)$.  The procedure is given in sufficient detail in Lemma \ref{ah}, which deals with a G\aa rding type inequality with a boundary term.  

The terms denoted by $\mathcal{G}$ in \ref{expl2} can be controlled w.r.t. well chosen norms by carrying out an integration by parts w.r.t. time $t$ and switching the time derivatives onto $a$ and using the fact that $t a''\in L^1(0,1)$ (see assumption $(a_3)$).  Eventually one upper bounds w.r.t. $L^\infty_t L^2_x$ norms $v$, $v_x$, $v_t$, $v_{xt}$, $v_{tt}$, and w.r.t. $L^2_t L^2_x$ norms $v$, $v_x$, $v_t$, $v_{xt}$.  The results are gathered into $\mathcal{E}_1$, see \eqref{imeq1}.  We point out that the aforementioned energy estimates do not provide norm estimates for $v_{xx}$.  To cope with this difficulty we use \eqref{expl2} which allows to express $v_{xx}$ as a function of $v_t$, $f$ and $\mathcal{G}$ with the help of an inversion Theorem for the operator  $w\mapsto a*w$ and using the previously obtained estimates.  We cannot use the resolvent kernel technique like in Brandon and Hrusa \cite{bd1} because in this paper case $r'\notin L^1(\mathbb{R})$ (as $a''\notin L^1(\mathbb{R}_+)$).  Because of 
that we prove a point-wise inversion Theorem for the convolution of $a$ assuming pretty weak constraints on $a$: see Theorem \ref{il}.

\section{Preliminaries.}\label{prl}

We shall frequently employ the following inequalities:

\beq\label{iq1}
|xy|\leq \mu x^2+\dfrac{1}{4\mu}y^2,\, x,y\in\mathbb{R}, \mu>0
\eeq

\beq\label{iq2}
\|F_1*F_2\|_{L^p(0,T)}\leq \|F_1\|_{L^1(0,+\infty)}\|F_2\|_{L^p(0,T)},
\eeq

The above is true for any $T>0$, $F_1\in L^1(0,+\infty)$, and $F_2\in L^p(0,T)$, with $p\geq1$.  Functions $F_1$ and $F_2$ are extended to $\mathbb{R}$ by $0$.

For any $T>0$, $w\in \ms{C}^0\left([0,T];L^2(\Omega) \right)$, $b\in L^1(0,+\infty)$ and $t\in[0,T]$.  We define

\begin{align}\label{qdf}
Q(w,t,b) & :=\int_0^t \int_\Omega w(x,s)\int_0^s b(s-\tau)w(x,\tau)\ud\tau \ud x \ud s \nonumber\\
& = \int_0^t \int_\Omega w(x,s)(b*w)(x,s)\ud x \ud s
\end{align}

where $w$ is considered as extended by $0$ on $(T,+\infty)$.  For any $T>0$ and $h\in(0,T)$, we define the finite difference operator $\Delta_h$

\beq\label{fd1}
\left( \Delta_h w \right) (x,t)=w(x,t+h)-w(x,t)
\eeq

as a linear operator from $\ms{C}^0\left([0,T-h];L^2(\Omega) \right)$ onto $\ms{C}^0\left([0,T];L^2(\Omega) \right)$.  




Moreover, if $X(J)$ denotes a space of functions defined on $J\subset \mathbb{R}$ and $I\subset J$, then $X_I(J)$ stands for the subspace of functions $X(J)$ the supports of which are included in $I$ (i.e. that vanish on $J-I$).

Recall that $b\in L^1\left(\mathbb{R}_+ \right) $ is of positive type if, for any $t\geq0$ and any $\f\in L^2\left(\mathbb{R}_+ \right) $, it satisfies $\displaystyle \int_0^t \f(s)\int_0^s b(s-\tau)\f(\tau)\ud\tau \ud s\geq0$.  Next, $b$ is said to be of strong positive type if there exists $\epsilon>0$ s.t. the function $b(t) - \epsilon e^{-t}$ is of positive type.  Moreover, $Q_t:=\Omega\times (0,t)$.

For future reference we prove the following Lemmas: 

\begin{lemma}\label{adl1}
Let the mappings $\f$ and $s\mapsto s\f(s)$ be elements of  $L^1\left(\mathbb{R}_+ \right)$.  Then the function $s\mapsto \displaystyle \int_s^{+\infty} \f(\tau)\ud\tau$ belongs to $L^1\left(\mathbb{R}_+ \right)$ and we have the estimate 

$$\displaystyle \int_0^{+\infty}\left|\int_s^{+\infty}\f(\tau)\ud\tau \right| \ud s\leq \int_0^{+\infty}\left| s\f(s) \right| \ud s$$

\end{lemma}

\begin{proof}
The proof is a direct consequence of  Fubini's Theorem.  



\end{proof}

\begin{lemma} \label{adl2}
Let $\f\in L^1\left(\mathbb{R}_+ \right)$.  Then:

\begin{enumerate}[(i)] 
\item for any $w_1,w_2\in L^2(Q_t)$ we have \label{adl22}
\beq\label{adl23}
\left|\int_0^t \int_{\Omega} w_1(x,s)(w_2*\f)(x,s)\ud s\right|\leq \|\f\|_{L^1\left(\mathbb{R}_+ \right)}\|w_1\|_{L^2(Q_t)}\|w_2\|_{L^2(Q_t)}
\eeq 
\item for any $w_3\in L^2(\Omega)$, $w_4\in L^\infty \left(0,T; L^2(\Omega) \right) $ we have \label{adl24}
\beq\label{adl25}
\left|\int_{\Omega} w_3(x)(\f*w_4)(x,t)\ud x\right|\leq \|\f\|_{L^1\left(0,T \right)}\|w_3\|_{L^2(\Omega)} \mathop{\sup}_{0\leq \tau \leq t}
\|w_4 (\tau)\|_{L^2(\Omega)},\; \text{a.e.}\, t\in [0,T)
\eeq
\end{enumerate}
\end{lemma}

\begin{proof}
Part \eqref{adl22}: observe that
\begin{align}\label{adl26}
\left|\int_0^t \int_{\Omega} w_1(x,s)(w_2*\f)(x,s)\ud s\right| & \leq \int_{\Omega} \left\|w_1(x,\cdot)\right\|_{L^2(0,t)} \left\|(w_2*\f)(x,\cdot)\right\|_{L^2(0,t)}\ud x \nonumber\\
& \leq \|\f\|_{L^1\left(\mathbb{R}_+ \right)}\int_{\Omega} \left\|w_1(x,\cdot) \right\|_{L^2(0,t)}\left\|w_2(x,\cdot) \right\|_{L^2(0,t)}\ud x
\end{align} 
which gives the result.

Part \eqref{adl24}: one has

\beq\label{adl27}
\left|\int_{\Omega} w_3(x)(\f*w_4)(x,t)\ud x\right|\leq \|w_3\|_{L^2(\Omega)}\int_0^t \|w_4(x,t-\tau)\|_{L^2(\Omega)}|\f(\tau)|\ud\tau
\eeq

and the result follows.
\end{proof}

We continue by proving the following result:

\begin{lemma}\label{lm1}
Assume $b\in W^{1,1}\left((0,+\infty) \right) $ verifies: there exists $M>0$ s.t. 
\beq\label{sp1}
\text{Re}\left[\mathcal{F}b (\omega) \right]\geq \dfrac{M}{1+\omega^2},\, \forall \omega\in\mathbb{R}
\eeq
Then:
\begin{enumerate}[(i)]
\item $b(0_+)\geq M$ \label{lmh1},
\item $\left|\mathcal{L}b(z) \right|\geq \dfrac{M}{2\left(1+|z|^2 \right) }$, $\forall z\in\mathbb{C}$, $\text{Re}(z)\geq0$\label{lmh2},
\item  $\left|\mathcal{F}b(\omega) \right|\geq \dfrac{\tilde{M}}{2\left(1+|\omega|\right) }$, $\forall \omega\in\mathbb{R}$\label{lmh3}, where $\tilde{M}$ may depend on $b$.
\end{enumerate}
 
\end{lemma}

\begin{proof}
Part \eqref{lmh1} is a direct consequence of 

$$b(0_+)=\dfrac{1}{\pi}\displaystyle \mathop{\lim}_{k\to+\infty}\int_{-k}^k \mathcal{F}b(\omega)\ud \omega=\dfrac{1}{\pi}\displaystyle \mathop{\lim}_{k\to+\infty}\int_{-k}^k \text{Re}\left[\mathcal{F}b(\omega) \right] \ud \omega $$ 

and of \eqref{sp1}

Part \eqref{lmh2}:  one has $\text{Re}\left[\left(\mathcal{F} e^{-t}\right)(\omega) \right]=\dfrac{1}{1+\omega^2}$.  This fact, together with  Theorem 2.4 on page 494 of \cite{grip1} imply that the function $t\in[0,+\infty)\mapsto b(t)-Me^{-t}$ is of positive type.  From the same Theorem one also gets $\text{Re}\left[\mathcal{L}\left(b-M e^{-t} \right) (z) \right] \geq0$, $\forall z\in\mathbb{C}$ with $\text{Re}(z)\geq0$.  The later in turn implies $\text{Re}\left[\mathcal{L}b (z)\right]\geq M \dfrac{1+z_1}{(1+z_1)^2+z_2^2}$, $\forall z=z_1+iz_2$ with $z_1,z_2\in\mathbb{R}$, $z_1\geq0$. The statement in \eqref{lmh2} now follows.  

Part \eqref{lmh3} is a consequence of \eqref{lmh2} and the fact that $b\in W^{1,1}(0,+\infty)$.  Indeed, from  $\left|\mathcal{F}b(\omega)\right|\geq\dfrac{M}{ 2(1+\omega^2) }$, $\forall \omega\in\mathbb{R}$, it suffices to prove that there exist $m_1,m_2>0$ s.t. $\left|\mathcal{F}b(\omega)\right|\geq\dfrac{m_1}{|\omega|}$, $\forall \omega\in\mathbb{R}$ with $|\omega|\geq m_2$.  This follows from $\mathcal{F}b(\omega)=\dfrac{1}{i\omega} \left[\mathcal{F}b'(\omega) +b(0_+)\right] $, the fact that $\displaystyle \mathcal{F}b'(\omega)\mathop{\longrightarrow}_{|\omega|\to+\infty}0$ and \eqref{lmh1}.
      
\end{proof}

The following Lemma is a G\aa rding type inequality with boundary terms.  It is proved in \cite{bd1} using preliminary results due to Staffans \cite{stf1} (see also \cite{gar1} and \cite{tay}).  Here we shorten the original proof of \cite{bd1} and remove the extraneous  assumptions $b\in W^{3,1}(0,+\infty)$, $b''\geq0$.   

\begin{lemma}\label{ah}
Assume $b\in L^1_{\mathbb{R}_+}\left(\mathbb{R} \right)$ is such that $\text{Re}\left(\hat{b}(\omega)  \right) \geq \dfrac{M_1}{1+\omega^2}$, for any $\omega\in \mathbb{R}$, where $M_1>0$.  Then, for any $T>0$, $w\in \mathscr{C}^1\left( [0,T], L^2(\Omega)\right) $  and $t\in [0,T)$, we have

\begin{align}\label{ah1}
& \int_\Omega w^2(x,t)\ud x + \int_0^t\int_\Omega w^2(x,s)\ud x \ud s  \nonumber\\
& \leq C \left[ \dfrac{1}{M_1}Q(w,t,b)+\dfrac{1}{M_1}Q(w_t,t,b)+ \int_\Omega w^2(x,0)\ud x\right]  
\end{align}

with $C>0$ independent of $T$, $t$, $w$ and $b$.  

Moreover, if $w\in \mathscr{C}^0\left( [0,T], L^2(\Omega)\right) $, then, for any $t\in [0,T]$,

\begin{align}\label{ah2}
& \int_\Omega w^2(x,t)\ud x + \int_0^t\int_\Omega w^2(x,s)\ud x \ud s  \nonumber\\
& \leq C \left[ \dfrac{1}{M_1}Q(w,t,b)+\dfrac{1}{M_1} \displaystyle\mathop{\liminf}_{h\to 0_+}  \dfrac{1}{h^2}Q(\triangle_h w,t,b)+ \int_\Omega w^2(x,0)\ud x\right]  
\end{align}

\end{lemma}

\begin{proof}
Assuming that inequality \eqref{ah1} holds true, we undertake to proving \eqref{ah2}.  Let  $w\in \mathscr{C}^0\left( [0,T], L^2(\Omega)\right) $ and $t\in [0,T)$ be fixed.  For $0<h<(T-t)/2$, define the function $w_h\in \mathscr{C}^1\left( \left[ 0,(t+T)/2\right] , L^2(\Omega)\right) $ by

\beq\label{ah3}
w_h(s):=\dfrac{1}{h}\int_{s}^{s+h}w(\sigma)\ud\sigma,\, s\in \left[ 0,(t+T)/2 \right) 
\eeq

Applying \eqref{ah1} to $w_h$ and passing to the limit $\displaystyle\mathop{\liminf}_{h\to 0_+}$ gives \eqref{ah2}.  

We now prove \eqref{ah1}.  Let $w\in \mathscr{C}^1\left( [0,T], L^2(\Omega)\right) $, $t\in[0,t)$ be fixed, and let $\tilde{w}\in L^2_{[0,t)}\left(\mathbb{R},L^2(\Omega) \right) $ be defined by $\tilde{w}=w$ a.e. in $[0,t]$ and $\tilde{w}=0$ outside.  Denote by $D \tilde{w}$ the distributional derivative of $\tilde{w}$ and by $\tilde{w}'$ its regular part, i.e. 

\beq\label{ah4}
D \tilde{w}=\tilde{w}'+w(0)\delta_0-w(t)\delta_t
\eeq

Due to the Parseval identity we have

\beq\label{ah5}
Q(w,t,b)=\dfrac{1}{2\pi} \int_\mathbb{R} \int_\Omega \text{Re}\left(\hat{b}(\tau) \right) \left|\widehat{\tilde{w}}(x,\tau)\right|^2\ud x\ud\tau
\eeq

and a similar equation with $w'$ instead of $w$ as well.  For $\lambda>0$ (to be later determined) define $I(w)$ by

\beq\label{ah6}
I(w):= Q\left(\tilde{w}',t,b \right) + \lambda Q\left(\tilde{w},t,b \right)+ \dfrac{3 M_1}{2}\int_\Omega w^2(x,0)\ud x
\eeq

By \eqref{ah4} and \eqref{ah5} and the strong positivity of $b$,

\begin{align}\label{ah7}
& I(w)\geq \nonumber\\
& \dfrac{M_1}{2\pi} \int_\mathbb{R} \int_\Omega \left( \left|i \tau \widehat{\tilde{w}}(\tau)-w(0) +w(t)e^{-i\tau t}  \right|^2 + \lambda\left| \widehat{\tilde{w}}(\tau)\right|^2+3|w(0)|^2 \right) \ud x \dfrac{\ud\tau}{1+\tau^2}
\end{align}

Since for any $(a,b,c)\in\mathbb{C}^3$ we have $|a+b+c|^2\geq \dfrac{|a|^2+|b|^2}{2}- 2|a||b| - 3|c|^2$, inequality \eqref{ah7} implies

\begin{align}\label{ah8}
& I(w)\geq \nonumber\\
& \dfrac{M_1}{2\pi} \int_\mathbb{R} \int_\Omega \left( \dfrac{|\tau|^2+2\lambda}{2}\left|\widehat{\tilde{w}}(\tau)\right|^2+|w(t)|^2 \beta \sqrt{|\tau|}-2|w(t)||\tau|\left|\widehat{\tilde{w}}(\tau)\right| \right) \ud x \dfrac{\ud\tau}{1+\tau^2}  
\end{align}

with 

\beq\label{ah9}
\beta=\displaystyle \dfrac{1}{2}\left(\int_\mathbb{R}\dfrac{\ud\tau}{1+\tau^2}  \right) /\left(\int_\mathbb{R}\dfrac{\sqrt{|\tau|}}{1+\tau^2}\ud\tau  \right)
\eeq 

But:

\begin{align} \label{ah10}
2 |w(t)| |\tau| \left|\widehat{\tilde{w}}(\tau)\right| & \leq \dfrac{\beta}{2} \sqrt{|\tau|} |w(t)|^2+ \dfrac{2}{\beta} |\tau|^{3/2} \left|\widehat{\tilde{w}}(\tau)\right|^2 \nonumber\\
& \leq \dfrac{\beta}{2} \sqrt{|\tau|} |w(t)|^2+\left(\dfrac{|\tau|^2}{4}+L \right)\left|\widehat{\tilde{w}}(\tau)\right|^2   
\end{align}

with $L>0$ independent of $t$, $w$, $b$.  Choose $\lambda=L+1/4$.  By \eqref{ah8} and \eqref{ah10} we get

\begin{align}\label{ah11}
& I(w)\geq 
\dfrac{M_1}{2\pi} \int_\mathbb{R} \int_\Omega \left( \dfrac{|\tau|^2+1}{4}\left|\widehat{\tilde{w}}(\tau)\right|^2+ \dfrac{\beta \sqrt{|\tau|}}{2}|w(t)|^2 \right) \ud x \dfrac{\ud\tau}{1+\tau^2}  
\end{align}

which is \eqref{ah1}.

\end{proof}

We now prove that, under suitable assumptions application $w\mapsto b*w$ is invertible, and obtain an inversion formula. We use truncated Neumann series and a special assumption (see $(b_3)$ below) in order to control the remainder term.   

For $b\in L^1(\mathbb{R})$, let the $k$-times convolution de denoted as $b^{*k}:=\underbrace{b*b*\cdots *b}_{k\,\text{times}}$.  For $1\leq q\leq +\infty$ and $t_0\in(0,+\infty]$, the mapping $\mathcal{R}_{t_0,q}$ is defined by:

\begin{displaymath}
\mathcal{R}_{t_0,q}:\left\{ \begin{array}{ll}
L^q_{[0,t_0)}(-\infty,t_0) & \longrightarrow  W^{1,q}_{[0,t_0)}(-\infty,t_0)\\                   
w & \mapsto b*w                   
\end{array}\right.
\end{displaymath}

Here $b*w(t):=\displaystyle \int_0^t b(t-s)w(s)\ud s$, for any $t<t_0$.  We always write $\mathcal{R}$ in place of $\mathcal{R}_{+\infty,2}$.    
 
Next, function $b$ is assumed to comply with:

\begin{enumerate}[$(\text{b}_1)$]\label{hl0}
\item $b\in W^{1,1}(0,+\infty)$, $b(0_+)\neq0$,\label{hl1}

\item\label{hl2} there exists $M>0$, $\beta>0$ s.t.
\beq\label{hl3}
\left|\mathcal{L}b(z) \right|\geq\dfrac{M}{1+|z|^\beta},\,\forall z\in\mathbb{C},\text{Re}(z)\geq0
\eeq

\item\label{hl4} there exists $p\in\mathbb{N}^*$, $p\geq2$ s.t.
\beq\label{hl5}
\mathcal{F}^{-1}\left[ \dfrac{\left( \mathcal{F}b'\right)^p}{\mathcal{F}b} \right]\in L^1(\mathbb{R}) 
\eeq
\end{enumerate}

Notice that $(\text{b}_1)$ and $(\text{b}_2)$ imply the following: there exists $M>0$ s.t.

\beq\label{hl7}
\left|\mathcal{F}b(\omega) \right|\geq \dfrac{M}{1+|\omega|},\,\forall \omega\in\mathbb{R}
\eeq

(see the proof of part (iii) in Lemma \ref{lm1}).

Our goal is to prove the following inversion Theorem:

\begin{theorem}[{\bf Inversion Theorem}]\label{il}
Let the assumptions  $(\text{b}_1)$ - $(\text{b}_3)$ hold true.  Then:

\begin{enumerate}[$(i)$]\label{il1}
\item\label{il2} for any $1\leq q \leq +\infty$ and $t_0\in(0,+\infty]$, the mapping $\mathcal{R}_{t_0,q}$
is a Banach isomorphism;
\item\label{il3} functions $B_1$, $B_2$ that depend only on $b$  and are being given by
\beq\label{il4}
B_1=\sum_{k=1}^{p-1} (-1)^k \dfrac{(b')^{*k}}{b^{k+1}(0_+)}
\eeq
\beq\label{il5}
B_2=\dfrac{(-1)^p}{b^{p}(0_+)}\mathcal{F}^{-1}\left[\dfrac{\left( \mathcal{F}b'\right)^p}{\mathcal{F}b} \right],
\eeq
belong to $L^1_{\mathbb{R}_+}(\mathbb{R})$;
\item\label{il6} for any $l\in W^{1,q}_{[0,t_0)}(-\infty,t_0)$, one has
\beq\label{il7}
\mathcal{R}^{-1}_{t_0,q}(l)=\dfrac{l'}{b(0_+)}+B_1*l'+B_2*l
\eeq
\end{enumerate} 
\end{theorem}

For the proof we first need to introduce and prove two preliminary Lemmas.

\begin{lemma}\label{pl1}
Assume that $b\in W^{1,1}\left( {\mathbb{R}^*_+}\right)$, $b(0_+)\neq0$.  Let $1\leq q\leq +\infty$, $t_0\in(0,+\infty)$.  Then $\mathcal{R}_{t_0,q}$ is a continuous injection.
\end{lemma}

\begin{proof}
We begin by showing $\mathcal{R}_{t_0,q}$ is well defined and continous.  Since $b\in W^{1,1}\left( {\mathbb{R}^*_+}\right)$, it is clear that for any $w\in L^q_{[0,t_0)}(-\infty, t_0)$, the function $b*w$ belongs to  $W^{1,q}_{[0,t_0)}(-\infty, t_0)$.  Moreover, $(b*w)'=\left[b(0_+)w+b'*w \right]$.  Hence 

\beq\label{pl12}
\displaystyle \left\| \mathcal{R}_{t_0,q}(w)\right\|_{W^{1,q}(0,t_0)}\leq \left[|b(0_+)|+\|b\|_{W^{1,1}\left( \mathbb{R}^*_+\right) } \right] \|w\|_{L^q(0,t_0)}
\eeq  

which proves $\mathcal{R}_{t_0,q}$ is indeed continous.

Next, assume $w\in L^q_{[0,t_0)}(-\infty, t_0)$ satisfies $\mathcal{R}_{t_0,q}(w)=0$.  Derivating the later leads to

\beq\label{pl13}
w(s)+\int_0^s \dfrac{b'(s-\tau)}{b(0_+)}w(\tau)\ud\tau =0,\, \text{a.e.}\, s<t_0
\eeq

Multiply \eqref{pl13} by $e^{-\theta s}$, $\theta>0$, and set $w_1(s)=e^{-\theta s}w(s)$, $b_1(s)=\dfrac{b'(s)}{b(0_+)}e^{-\theta s}$.  Equality \eqref{pl13} can now be re-written as

\beq\label{pl14}
w_1(s)+\int_0^t b_1(s-\tau)w_1(\tau)\ud\tau =0,\, \text{a.e.}\, s<t_0
\eeq

It implies that 

\beq\label{pl15}
\|w_1\|_{L^q(0,t_0)}\leq \|b_1\|_{L^1\left( \mathbb{R}^*_+\right)}\|w_1\|_{L^q(0,t_0)}
\eeq

Notice that  $\displaystyle \|b_1\|_{L^1}=\int_0^{+\infty}e^{-\theta s}\dfrac{|b'(s)|}{|b(0_+)|}\ud s\mathop{\longrightarrow}_{\theta\to+\infty}0$.  Pick up a $\theta>0$ large enough s.t. 

$\|b_1\|_{L^1\left( \mathbb{R}^*_+\right)}<1$.  From \eqref{pl15} we get $\|w_1\|_{L^1(0,t)} =0$.  Finally $w=0$ and $\mathcal{R}_{t_0,q}$ is an injection mapping.  

\end{proof}


\begin{lemma}\label{pl2}

The Theorem \ref{il} holds true for $t_0=+\infty$ and $q=2$.






\end{lemma}

\begin{proof}
The proof consists of three steps.  

{\bf Step 1.}

First we prove $\mathcal{R}$ is a Banach isomorphism.  Due to Lemma \ref{pl1}, one only needs to prove 
$\mathcal{R}$ is surjective.  To begin with, one establishes that, for any 
$w\in  L^2_{\mathbb{R}_+}(\mathbb{R})$, one has (with $M>0$ the constant in \eqref{hl7})

\beq\label{pl24}
\|w\|_{L^2\left( \mathbb{R}\right) }\leq \dfrac{1}{\sqrt{\pi}M}\|\mathcal{R}(w)\|_{H^1\left( \mathbb{R}\right)}
\eeq

Actually using Parseval's identity and \eqref{hl7} one gets

\beq\label{pl25}
\sqrt{2\pi}\|w\|_{L^2\left( \mathbb{R}\right)}= \left\|\mathcal{F}w \right\|_{L^2\left( \mathbb{R}\right)}=\left\|\dfrac{\mathcal{F}\mathcal{R}(w)}{\mathcal{F}b}   
\right\|_{L^2\left( \mathbb{R}\right)}\leq \dfrac{1}{M}\left\|(1+|\omega|)\mathcal{F}\mathcal{R}(w) \right\|_{L^2\left( \mathbb{R}\right)}
\eeq

Since $(1+|\omega|)\leq \sqrt{2(1+\omega^2)}$, inequality \eqref{pl25} implies inequality \eqref{pl24}.  Next, inequalities \eqref{pl12} and \eqref{pl24} prove that $\mathcal{R}\left(L^2_{\mathbb{R}_+}(\mathbb{R}) \right) $ is closed.  Therefore, in order to prove  that $\mathcal{R}$ is surjective it is sufficient to show that the dense subset $\left( \mathscr{C}^\infty_c\right)_{(0,+\infty)}(\mathbb{R})$ of  $H^1_{\mathbb{R}_+}(\mathbb{R})$ is included in $\mathcal{R}\left(L^2_{\mathbb{R}_+}(\mathbb{R}) \right) $.  

Let $r\in \left( \mathscr{C}^\infty_c\right)_{(0,+\infty)}(\mathbb{R})$.  We search for $w\in L^2_{\mathbb{R}_+}(\mathbb{R})$ s.t. $b*w=r$.  Since we are unable to identify the support of $w$ by Fourier transform, we use Laplace transform instead.  Consider the function

$$z\in \left\lbrace z\in\mathbb{C}/ \,\text{Re}(z)\geq0 \right\rbrace \mapsto\dfrac{\mathcal{L}r(z)}{\mathcal{L}b(z)}\in\mathbb{C}$$

which is well defined based on $(b_2)$ and the fact that $r\in \left( \mathscr{C}^\infty_c\right)_{(0,+\infty)}(\mathbb{R})$.  This function is clearly continuous on $\text{Re}(z)\geq0$ and analytic on $\text{Re}(z)>0$.  As for any $z\in\mathbb{C}$ and $\gamma\in\mathbb{N}$, $\mathcal{L}r^{(\gamma)} (z)=z^\gamma  \mathcal{L}r(z)$, and as $r^{(\gamma)}\in L^1(\mathbb{R})$, we deduce that there  exists $m_1\geq0$ s.t. 

$$\left|\mathcal{L}r(z) \right| \leq \dfrac{m_1}{1+|z|^{\beta+2}},\, \forall  z\in\mathbb{C},\,\text{Re}(z)\geq0$$

Now it easily follows the existence of $m_2\geq0$ s.t. 

\beq\label{pl26}
\left| \dfrac{\mathcal{L}r(z)}{\mathcal{L}b(z)}\right| \leq \dfrac{m_2}{1+|z|^2},\, \forall  z\in\mathbb{C},\,\text{Re}(z)\geq0
\eeq

Next, with the help of Bromwich-Mellin formula, for any $t\in\mathbb{R}$ and for fixed $x>0$,   define $w$ as

\beq\label{pl27}
w(t):=\dfrac{1}{2\pi i}\int_{\mathbb{R}}e^{t(x+iy)}\dfrac{\mathcal{L}r}{\mathcal{L}b}(x+iy)\ud y
\eeq

Owing to Cauchy's formula and invoking \eqref{pl26}, $w$ thus defined is independent of $x>0$.  
Also, for fixed $t<0$, letting $x\to+\infty$ in \eqref{pl27} leads to $w(t)=0$.  This is   $w(t)=0$ for any $t<0$.  Next, for any fixed $t\in\mathbb{R}$, using Lebesgue's Theorem we calculate the limit for $x\to0$ of \eqref{pl27} and obtain $w=\mathcal{F}^{-1}\left(\dfrac{\mathcal{F}r}{\mathcal{F}b} \right) $.  By Parseval's identity and by \eqref{pl26}, $w$ is clearly an element of $L^2_{\mathbb{R}_+}(\mathbb{R})$ and satisfies $\mathcal{R}(w)=r$.  Therefore $\mathcal{R}$ is surjective.    

{\bf Step 2.}

The task now is proving the representation formula.  Let $w\in L^2_{\mathbb{R}_+}(\mathbb{R})$ and set $l=\mathcal{R}(w)$. Derivation of the later gives

\beq\label{pl28}
w+\dfrac{b'}{b(0_+)}*w=\dfrac{l'}{b(0_+)}
\eeq 

Convolute \eqref{pl28} with the operator $\displaystyle \sum_{k=0}^{p-1}(-1)^k \left(\dfrac{b'}{b(0_+)} \right) ^{*k}*$ (by convention $\left(\dfrac{b'}{b(0_+)} \right) ^{*0}=\delta_0$).  We obtain:

\beq\label{pl29}
w=\dfrac{l'}{b(0_+)}+\left(B_1*l' \right)+\dfrac{(-1)^p}{b^p(0_+)} \left[(b')^{*p}*w \right] 
\eeq

Since $l=b*w$, we get $\mathcal{F}l=\mathcal{F}b\mathcal{F}w$.  Hence 

\beq\label{pl230}
\mathcal{F}\left[(b')^{*p}*w \right]=\left( \mathcal{F}b'\right) ^{p}\dfrac{\mathcal{F}l}{\mathcal{F}b}
\eeq

By hypothesis $(b_3)$, $\dfrac{\left( \mathcal{F}b'\right) ^{p}}{\mathcal{F}b}\in L^\infty(\mathbb{R})$, which proves that inequality \eqref{pl230} holds in $L^2(\mathbb{R})$ since $\mathcal{F}l\in L^2(\mathbb{R})$.  This fact allows to state that $\dfrac{(-b')^{*p}}{b^p(0_+)}*w=B_2*l$ with $B_2$ given by \eqref{il5}.  Now, for any $w\in L^2_{\mathbb{R}_+}(\mathbb{R})$ and $l=\mathcal{R} (w)$,  \eqref{pl29} gives the representation formula 

\beq\label{pl231}
w=\dfrac{l'}{b(0_+)}+B_1*l'+B_2*l
\eeq

{\bf Step 3.}

Let us now show that the support of $B_1$ and that of $B_2$ are included in $\mathbb{R}_+$. 

Since the support of $b'$ is in $\mathbb{R}_+$, $B_1$ also has its support in $\mathbb{R}_+$ due to formula \eqref{il4}.  Let $\rho\in \mathscr{D}_{\mathbb{R}_+}\left(\mathbb{R} \right) $ and set $w=\mathcal{R}^{-1}(\rho)$ (see \textbf{Step 1.}).  Equation \eqref{pl231} now ensures that, a.e. $t<0$,

\beq\label{pl232}
0=w(t)=\dfrac{\rho'(t)}{b(0_+)}+\left( B_1*\rho'\right) (t)+\left( B_2*\rho\right) (t)
\eeq

Since $\rho'(s)=0$ a.e. $s<0$ and since $B_1$ has support in $\mathbb{R}_+$, we get

\beq\label{pl233}
\left( B_2*\rho\right) (t)=0,\,\text{a.e.}\, t<0
\eeq

Take $\rho\geq0$, $\rho\neq0$, and set $\rho_n(t)=n\rho\left(nt \right) $, $n\in\mathbb{N}^*$, $t\in \mathbb{R}$.  We know that:

\beq\label{pl234}
\displaystyle B_2*\rho_n \xrightarrow[n\to+\infty]{L^1(\mathbb{R})}\|\rho\|_{L^1} B_2
\eeq

Taking $\rho=\rho_n$ in \eqref{pl233} and using \eqref{pl234} we obtain $B_2=0$ a.e. $t<0$.  Hence $B_2$ has support in $\mathbb{R}_+$.

\end{proof}

We are now in a position allowing to prove the previously stated Inversion Theorem \ref{il}.

\begin{proof} {\bf Proof of the Inversion Theorem \ref{il1}}

Let $q\in [1,+\infty)$ and $t_0\in \mathbb{R}^*_+ \cup \{+\infty\}$.  Define the mapping $\mathcal{S}_{t_0,q}$ by:

\begin{displaymath}
\mathcal{S}_{t_0,q}=\left\{ \begin{array}{ll}
W^{1,q}_{[0,t_0)}(-\infty,t_0) & \longrightarrow  L^{q}_{[0,t_0)}(-\infty,t_0)\\                   
l & \mapsto \dfrac{l'}{b(0_+)}+B_1*l'+B_2*l                   
\end{array}\right.
\end{displaymath}

with $B_1, B_2\in L^1_{\mathbb{R}_+}(\mathbb{R})$ given by \eqref{il4}-\eqref{il5}.  Clearly $\mathcal{S}_{t_0,q}$ is well defined and continuous.

We begin by studying the case $t_0=+\infty$.  

Notice that $\mathcal{S}_{+\infty,q}\circ \mathcal{R}_{+\infty,q}$ restricted to $D=L^q_{\mathbb{R}_+}(\mathbb{R})\cap L^2_{\mathbb{R}_+}(\mathbb{R})$ is the identity (see Lemma \ref{pl2}).  Since $D$ is dense in $ L^q_{\mathbb{R}_+}(\mathbb{R})$, and 
$\mathcal{S}_{+\infty,q}$ and $\mathcal{R}_{+\infty,q}$ are continuous, we find that $\mathcal{S}_{+\infty,q}\circ \mathcal{R}_{+\infty,q}$ is the identity on $ L^q_{\mathbb{R}_+}(\mathbb{R})$.  Similarly, $\mathcal{R}_{+\infty,q}\circ \mathcal{S}_{+\infty,q}$ is the identity on $ W^{1,q}_{\mathbb{R}_+}(\mathbb{R})$. This proves the Theorem for $t_0=+\infty$.

Assume now that $t_0>0$ and $q\in [1,+\infty]$.  We know from Lemma \ref{pl1} that 
$\mathcal{R}_{t_0,q}$ is continuous and injective.  We now prove that $\mathcal{R}_{t_0,q}$ is surjective and 
that $\mathcal{S}_{t_0,q}$ is its inverse.   Let $l\in W^{1,q}_{[0,t_0)}(-\infty,t_0)$ and extend $l$ into $L\in W^{1,q}_{[0,2t_0)}(\mathbb{R})$ by reflexion 

\begin{displaymath}
L(t)=\left\{ \begin{array}{ll}
l(t) & \,\text{for}\,  t<t_0\\                   
l(2t_0-t) & \,\text{for}\, t>t_0                 
\end{array}\right.
\end{displaymath}

Let $W=(\mathcal{S}_{+\infty,q})(L)$ and define $w\in L^{q}_{[0,t_0)}(-\infty,t_0)$ as the restriction of $W$ to $(-\infty,t_0)$.  Then, $b*w=b*W=l$ on $(-\infty,t_0)$, and:

\beq\label{il8}
w=W=\dfrac{L'}{b(0_+)}+B_1*L'+B_2*L,\,\text{on}\,(-\infty,t_0) 
\eeq

This is $w=\mathcal{S}_{t_0,q}(l)$.  This proves the Theorem.
 
\end{proof}

Notice that from hypotheses $(a_1)$, $(a_4)$, $(a_5)$ and Lemma \ref{lm1},  the above Inversion Theorem can be used with $b=a_n$.

\section{Approximated problems and estimates.}\label{appe}

\subsection{Approximated and local problems.  Preliminary notations and estimates.}

Remark that $a$ is not smooth enough to ensure a straightforward  local in time existence result for a solution $v$ to our problem.   As a consequence we study the following approximated problem which we denote by $P_n$.  

{\bf Problem ${\bf P_n}$}: find  $v_n:\Omega\times [0,+\infty)\to \mathbb{R}$ s.t.

\begin{enumerate}[$(P_n)_1$]\label{ap1}
\item $\displaystyle (v_n)_t=\int_0^{+\infty} a'_n(s)\dfrac{\partial}{\partial x}g\left(\left(\overline{v}^t_n \right)_x\right) (x,s)   \ud s +f(x,t)$ \label{ap2} 
\item $v_n=0$ on $\partial \Omega$, $v_n(t)=0,\, \forall t<0$ \label{ap3}
\item $v_n(x,0)=v_0(x)$ for $x\in \Omega$\label{ap4}
\end{enumerate}

Given the assumptions on $g$ we conclude there exist $\gamma>0$ and $\theta\in[0,1]$ s.t.

\beq\label{ap5}
g'(y)<-\gamma,\, \forall y\in [-\theta, \theta]
\eeq

Clearly we can take the same $\theta$ as in assumption $(g_1)$.  Moreover, there exists $K>0$ s.t. 



\beq\label{ap7}
\left|g'(y)-g'(0) \right|\leq Ky^2,\, \forall y\in [-\theta,\theta]
\eeq

In the above one may consider the same $K$ as in $(g_1)$.  

Let us denote, for almost every $x\in\Omega$,
$$
u_n(x,t) = \int_0^t v_n(x,s) \, ds.
   $$

The proof of the next Proposition is very similar to that of Theorem III.10 in \cite{rhn} and is omitted. 

\begin{proposition}\label{pr}
Assume that the hypotheses $(g_1)$-$(g_3)$, $(f_1)$-$(f_4)$, $(v_0)$, and $(a_1)$-$(a_5)$ on the data hold true.  
Then the initial value problem $(P_n)_1$, $(P_n)_2$, $(P_n)_3$ 
has a unique solution $v_n$ defined on a maximal time interval 
$[0,T_n)$, $T_n>0$, and s.t. 
$ v_n \in \mathscr{C}^0\left([0,T_n); H^2(\Omega) \right), \; 
(v_n)_t \in \mathscr{C}^0\left([0,T_n); H^1(\Omega) \right),
(v_n)_{tt} \in \mathscr{C}^0\left([0,T_n); L^2(\Omega) \right)$ and
$u_n \in \mathscr{C}^0\left([0,T_n); H^3(\Omega) \right).$
%
Moreover, if 

\beq\label{pr1}
\displaystyle\mathop{\sup}_{t\in[0,T_n)}  \left\lbrace \|v_n(\cdot, t)\|^2_{H^2(\Omega)}
+ \|(v_n)_t(\cdot, t)\|^2_{H^1(\Omega)} + \|(v_n)_{tt} (\cdot, t)\|^2_{L^2(\Omega)}
+ \|u_n(\cdot, t)\|^2_{H^3(\Omega)}
\right\rbrace  <\infty
\eeq

and

\begin{displaymath}\label{pr01}
\sup_{\substack{x\in\Omega \\ 0\leq t\leq T_n} }\left|\left( u_n\right)_x (x,t) \right|\leq \dfrac{\theta}{2}
\end{displaymath}

with $\theta$ as in $(g_4)$, then $T_n=+\infty$.  

\end{proposition}

Notice that our functional framework is different from that of \cite{bh1}.  As a consequence, here it is necessary to obtain new estimates on $\|u_n\|_{H^3(\Omega)}$.  

In this Section we obtain the necessary estimates to proving $T_n=+\infty$.  These estimates will be proved to be independent of $n$, fact which allows to pass to the limit as $n\to+\infty$.  To simplify notations, we drop the subscript $n$ of $a_n$, $v_n$ and $T_n$.  

Drawing inspiration from \cite{bh1}, we introduce the following expressions:

\begin{align}\label{pr2}
\mathcal{E}(t) & = \displaystyle\mathop{\sup}_{s\in[0,t)}\left[ \int_\Omega 
\left(v^2+v^2_{x}+v^2_{t}+v^2_{xx}+v^2_{xt}+v^2_{tt} 
+ u^2 + u_x^2 + u_{xx}^2 + u_{xxx}^2\right)(x,s)\ud x \right] \nonumber\\
& + \int_0^t \int_\Omega \left(v^2+v^2_{x}+v^2_{t}+v^2_{xx}+v^2_{xt}+v^2_{tt} \right)(x,s)\ud x\ud s
\end{align}

and 

\beq\label{pr3}
\nu(t)=\sup_{\substack{x\in\Omega \\ s\in[0,t]} }\left[ \sqrt{\left( v^2+v^2_x+v^2_t\right) (x,s)}\right] +\sqrt{\int_0^t \displaystyle\mathop{\sup}_{x\in\Omega}\left(v_x(x,s) \right)^2\ud s }
\eeq

For simplicity let us denote
\begin{align}\label{imeq1}
\mathcal{E}_1(t) = & \displaystyle\mathop{\sup}_{s\in[0,t)}\left[ \int_\Omega 
\left(v^2+v^2_{x}+v^2_{t}+v^2_{xt}+v^2_{tt} \right)(x,s)\ud x\right]  + \nonumber\\
& \int_0^t \int_\Omega \left(v^2+v^2_{x}+v^2_{t}+v^2_{xt} \right)(x,s)\ud x\ud s
\end{align}
In fact $\mathcal{E}_1(t)$ collects the terms of $\mathcal{E}(t)$ which will be estimated
in a first step with the help of energy estimates.

Remark that, due to Sobolev inequalities, there exists a constant $C_\Omega>0$ s.t.

\beq\label{pr4}
\nu(t)\leq C_\Omega \sqrt{\mathcal{E}(t)},\, \forall t\in [0,T) 
\eeq

and 

\beq\label{pr41}
\displaystyle\mathop{\sup}_{x\in \Omega}\left|u_x(x,t) \right|\leq C_\Omega \sqrt{\mathcal{E}(t)},\, \forall t\in [0,T)
\eeq

Next, from \eqref{p6} we get

\begin{align}\label{pr5}
\mathcal{G}_t(x,t) & = v_{xx}(x,t)\int_0^{+\infty}a'(s) \left[g'\left(\overline{v}_x^t(x,s) \right) -g'(0) \right]\ud s \nonumber\\
& -\int_0^t v_{xx}(x,s)a'(t-s)\left[g'\left(\overline{v}_x^t(x,t-s) \right) -g'(0) \right]\ud s \nonumber\\
& + \int_0^t v_{xx}(x,s)\int_{t-s}^{+\infty}a'(\tau)g''\left(\overline{v}_x^t(x,\tau) \right)\left[v_x(x,t)-v_x(x,t-\tau)  \right]\ud\tau\ud s \nonumber\\
\end{align}

All subsequent estimates will be obtained under the following smallness hypothesis on $\mathcal{E}(t)$: 

\beq\label{pr6}
\mathcal{E}(t) \leq \dfrac{\theta^2}{4 C^2_\Omega},\, \forall t\in[0,T)
\eeq

which implies 

\beq\label{pr61}
\sup_{\substack{x\in\Omega \\ 0\leq t \leq T} } \left|u_x(x,t) \right|\leq \dfrac{\theta}{2}
\eeq

Then

\beq\label{pr62}
\sup_{\substack{x\in\Omega \\ 0\leq s\leq t \leq T} }\left| \overline{v}_x^t(x,s) \right|\leq \theta,\,\text{a.e.}\, x\in\Omega
\eeq

Let $r_0:\mathbb{R}_+\to \mathbb{R}_+$, $r_0(s):=\displaystyle \min \left\lbrace s,\sqrt{s}\right\rbrace $.  We have the following estimates:

\begin{lemma}\label{le}
Let $t\in[0,T)$, assume \eqref{pr6} is satisfied. Then:

\begin{enumerate}[(i)]
\item $\displaystyle \left| g^{(j)}\left(\overline{v}_x^t(x,s) \right)-g^{(j)}(0) \right| \leq K \min \left\lbrace \nu(t)r_0(s),\theta \right\rbrace \,\text{a.e.}\, x\in\Omega,\, s\in[0,t],\,j=0,1,2,3$ \label{le1} 
\item \label{le2} $\left|\mathcal{G}(x,t) \right|\leq K \nu(t) \left[\left|v_{xx}(x,\cdot)\right|*\psi \right](t),\,\text{a.e.}\, x\in\Omega$  
\item \label{le3} $\left|\mathcal{G}_t(x,t) \right|\leq K \nu(t)\overline{a}\left|v_{xx}(x,t)\right|+K \nu(t) \left[\left|v_{xx}(x,\cdot)\right|*\psi \right](t),\,\text{a.e.}\, x\in\Omega  $, 
\end{enumerate}
where
\beq\label{le4}
\overline{a}=\int_0^{+\infty}\left|a'(s)\right|r_0(s)\ud s 
\eeq
\beq\label{le5}
\psi(t)=\left|a'(t)\right|r_0(t)+2\int_t^{+\infty}\left|a'(\tau)\right|r_0(\tau)\ud\tau
\eeq
\end{lemma}

\begin{remark}\label{rk1}
Lemma \ref{adl1} and the assumptions made about function $a$ grant the fact that $\psi$ in \eqref{le5} is s.t. $\psi\in L^1\left(\mathbb{R}_+ \right) $.
\end{remark}

\begin{proof}

(i)  On one hand, as a consequence of $(g_1)$ and \eqref{pr62} we have

\beq\label{le6}
\displaystyle \left| g^{(j)}\left(\overline{v}_x^t(x,s) \right)-g^{(j)}(0) \right|\leq K \left|\overline{v}_x^t(x,s)\right|,\, j=0,1,2,3
\eeq

On the other hand, 

\beq \label{le7}
\left|\overline{v}_x^t(x,s) \right| \leq \int_{t-s}^{t}\left| v_x(x,\lambda) \right|\ud\lambda \leq s \displaystyle \mathop{\sup}_{t-s\leq \lambda\leq t}\left| v_x(x,\lambda) \right|\leq s \nu(t)
\eeq

and

\beq\label{le8}
\left|\overline{v}_x^t(x,s) \right|\leq \sqrt{s} \left[\int_{t-s}^{t}\left|v_x(x,\lambda)  \right|^2\ud\lambda \right]^{1/2} \leq \sqrt{s} \nu(t)
\eeq

which gives the result.

(ii)  From \eqref{p6} and \eqref{le1} above one gets:

\begin{align}\label{le9}
\left|\mathcal{G}(x,t)\right| & \leq K \nu(t) \int_0^t \left|v_{xx}(x,s)\right|\int_{t-s}^{+\infty}a'(\tau)\min \{\tau,\sqrt{\tau}\}\ud\tau\ud s \nonumber\\
& \leq K \nu(t) \int_0^t\left|v_{xx}(x,s)\right|\psi(t-s)\ud s
\end{align}

from which the result follows.

(iii) We use \eqref{pr5}, $(g_1)$, \eqref{le1}, the fact that $g''(0)=0$ and $0\leq\theta\leq1$ to obtain:

\begin{align}\label{le10}
\left|\mathcal{G}_t(x,t)\right| & \leq K \left|v_{xx}(x,t)\right| \nu(t) 
\int_0^{+\infty}|a'(s)|r_0(s)\ud s  \nonumber\\
& + K \nu(t) \int_0^t \left|v_{xx}(x,s)\right|\left|a'(t-s)\right|r_0 (t-s) \ud s\nonumber\\
& +2K \theta \nu(t) \int_0^t \left|v_{xx}(x,s)\right| \int_{t-s}^{+\infty}|a'(\tau)|r_0(\tau)\ud\tau \ud s
\end{align}

which gives the result.

\subsection{Energy estimates.}

\end{proof}

The next Lemmas give energy estimates for the terms in $\mathcal{E}_1(t)$ (see \eqref{imeq1}), as in \cite{bh1}.


In what follows, the notation $C>0$ stands for a generic constant that is independent of $n$.  

\begin{lemma}\label{eel}
Assume the inequality \eqref{pr6} holds true.  Then
\beq\label{eel1} 
\int_{\Omega}v^2(x,t)\ud x-2g'(0)Q\left(v_x,t,a \right) \leq V_0+2\sqrt{F}\sqrt{\mathcal{E}(t)}+2K \|\psi\|_{L^1(\mathbb{R}_+)}\nu(t)\mathcal{E}(t)
\eeq
\end{lemma}

\begin{proof}
For a fixed $t\in (0,T_0)$, we multiply \eqref{p5} by $v(x,t)$ and integrate on $\Omega$ and on $(0,t)$.  We get

\begin{align}\label{eel2}
& \dfrac{1}{2}\int_{\Omega}v^2(x,t)\ud x- \dfrac{1}{2}\int_{\Omega} v_0^2\ud x-g'(0) Q(v_x,t,a) \nonumber\\
& = \int_0^t \int_{\Omega} f(x,s) v(x,s)\ud x \ud s + \int_0^t \int_{\Omega} \mathcal{G}(x,s) v(x,s) \ud x \ud s
\end{align}

Observe that $\displaystyle\int_0^t \int_{\Omega} fv\ud x \ud s \leq \|f\|_{L^2(Q_t)}\|v\|_{L^2(Q_t)}\leq \sqrt{F}\sqrt{\mathcal{E}}$. 

Now, using Lemma \ref{le} we get 

$$\displaystyle\left|\int_0^t \int_{\Omega} \mathcal{G}(x,s)v(x,s)\ud x\ud s\right| \leq K \nu(t)\int_0^t \int_{\Omega} \left| v(x,s) \right|\left(\left| v_{xx} \right|*|\psi| \right)(x,s) \ud x\ud s$$

Using part (i) of Lemma \ref{adl2} with $w_1=v$, $w_2=v_{xx}$ and $\f=|\psi|$ one gets 

$$\displaystyle\left|\int_0^t \int_{\Omega} \mathcal{G}(x,s)v(x,s)\ud x\ud s\right| \leq K \nu(t) \|\psi\|_{L^1(\mathbb{R}_+)} \mathcal{E}(t),$$ 

thus ending the proof.
\end{proof}

\begin{lemma}\label{eel3}
Let $\overline{a}$ and $\psi$ be given by \eqref{le4} and \eqref{le5}, respectively.  Under the assumption that \eqref{pr6} is fulfilled , one has the following inequality:

\begin{align}\label{eel4}
\int_\Omega v_t^2(x,t)\ud x -2g'(0) Q\left( v_{xt},t,a\right) & \leq F + 2\|a\|_{L^1(\mathbb{R}_+)}\sqrt{V_0}\sqrt{\mathcal{E}(t)} \nonumber\\
& + 2\sqrt{F}\sqrt{\mathcal{E}(t)} +2K \left( \|\psi\|_{L^1(\mathbb{R}_+)} +\overline{a}\right) \nu(t)\mathcal{E}(t)
\end{align}
\end{lemma}

\begin{proof}
First, we derivate \eqref{p5} w.r.t. $t$ and obtain

\beq\label{eel5} 
v_{tt}(x,t)+g'(0)a(0)v_{xx}(x,t)+g'(0)\int_0^t a'(t-s)v_{xx}(x,s)\ud s=f_t+\mathcal{G}_t
\eeq

Secondly, multiplying the above by $v_t$ and integrating on $\Omega$ and on $[0,t]$ leads to

\begin{align}\label{eel6}
& \dfrac{1}{2} \int_\Omega v_t^2(x,t)\ud x -\dfrac{1}{2} \int_\Omega v_t^2(x,0)\ud x-g'(0)a(0)\int_0^t\int_\Omega v_x v_{xt} \ud x \ud s \nonumber\\
& -g'(0)\int_0^t\int_\Omega \int_0^s a'(s-\tau) v_x(\tau)\ud\tau  v_{xt}(s)\ud x\ud s 
=  \int_0^t\int_\Omega f_t v_t \ud x \ud s + \int_0^t\int_\Omega \mathcal{G}_t v_t \ud x \ud s
\end{align}

Observe now that

\beq\label{eel7}
\int_0^s a'(s-\tau)v_x(\tau)\ud\tau=-a(0) v_x(s)+a(s)v_x(0)+\int_0^s a(s-\tau)v_{xt}(\tau)\ud \tau
\eeq

One now gets:

\begin{align}\label{eel8}
& \dfrac{1}{2}\int_\Omega v_t^2(x,t)\ud x- 
g'(0)Q\left(v_{xt},t,a \right)=\dfrac{1}{2} \int_\Omega  v_t^2(x,0)\ud x 
-g'(0) \int_0^t \int_\Omega a(s) v''_0(x)v_t(x,s)\ud x\ud s \nonumber\\
& +\int_0^t \int_\Omega \left( f_t v_t\right) (x,s) \ud x\ud s+\int_0^t \int_\Omega \left( \mathcal{G}_t v_t\right)(x,s)  \ud x \ud s
\end{align}

Notice that 
\beq\label{eel9}
v_t(x,0)=f(x,0)
\eeq

which gives $\displaystyle\int_\Omega  v_t^2(x,0)\ud x\leq F$.  We also have

\begin{align}\label{eel10}
\left|\int_0^t \int_\Omega a(s) v''_0(x)v_t(x,s)\ud x\ud s \right| & \leq \left\|v''_0\right\|_{L^2(\Omega)}\left\|a\right\|_{L^1(\mathbb{R}_+)}\mathop{\sup}_{0\leq s \leq t}\left\|v_t(\cdot,s)\right\|_{L^2(\Omega)}\nonumber\\
& \leq \left\|a\right\|_{L^1(\mathbb{R}_+)} \sqrt{V_0} \sqrt{\mathcal{E}(t)}
\end{align}

and
\beq\label{eel11}
\int_0^t \int_\Omega \left( f_t v_t\right) (x,s) \ud x\ud s \leq \sqrt{F}\sqrt{\mathcal{E}(t)}
\eeq

Finally, invoking part (iii) of Lemma \ref{le} and part (i) of Lemma \ref{adl2} we deduce that

\beq \label{eel12}
\int_0^t \int_\Omega \left( \mathcal{G}_t v_t\right)(x,s)  \ud x \ud s \leq K \overline{a} \nu(t) 
\mathcal{E}(t) + K \nu(t) \|\psi\|_{L^1(\mathbb{R}_+)}\mathcal{E}(t) 
\eeq

and with the obtainment of this last estimates the proof ends.  

\end{proof}

Next, in order to obtain energy estimates for $\displaystyle \int_\Omega v^2_{tt}(x,t)\ud x$ we shall use the difference operator $\displaystyle \left(\triangle_h w \right)(x,t)=w(x, t+h)-w(x,t)$, for $h>0$ small enough.


\begin{lemma}\label{lies}
Under the assumption that \eqref{pr6} is fulfilled,  one has:

\begin{align}\label{lies1}
\int_\Omega v_{tt}^2(x,t)\ud x -2g'(0) \displaystyle \mathop{\lim}_{h\to 0_+}\dfrac{1}{h^2} 
Q \left(\triangle_h v_{xt},t,a \right) \leq &  C \bigg\{ F + \sqrt{F}\sqrt{\mathcal{E}(t)} \nonumber\\
& + \left[\nu(t)+\nu^3(t) \right] \mathcal{E}(t) + \sqrt{V_0} \mathcal{E}(t) \bigg\}
\end{align}

\end{lemma}

For the Proof, see the Appendix Section.



Since $\nu(t)$ and $\mathcal{E}(t)$ are non-increasing functions in $t$, we obtain as a consequence 
of Lemma \ref{eel}, Lemma \ref{eel3}, Lemma \ref{lies}, Lemma \ref{ah} and Sobolev 
embeddings, that:

\begin{lemma}\label{lse} 
Under the assumption stated in \eqref{pr6} one has

\begin{equation}
\label{lse1}
\mathcal{E}_1(t) \leq C \left\lbrace V_0+F+\left(\sqrt{V_0}+\sqrt{F} \right)
\sqrt{\mathcal{E}(t)}+\left[\nu(t)+\nu^3(t) \right]\mathcal{E}(t) + \sqrt{V_0}\mathcal{E}(t) 
\right\rbrace 
\end{equation}


\end{lemma}

\subsection{Non-energy estimates.}

In the following we obtain estimates for the other constitutive terms of $\mathcal{E}(t)$.  

Now, from \eqref{p5} and using for a.e. $x\in\Omega$ the result of Theorem \ref{il} with $b=a$, 

$l(t)=\dfrac{1}{g'(0)}\left[f(x,t)+\mathcal{G}(x,t)-v_t(x,t) \right] $, and $w(t)=v_{xx}(x,t)$, we deduce the equality

\beq\label{sl0}
v_{xx}=\dfrac{1}{g'(0)}\left[ 
\dfrac{1}{a(0)}\left(f_t+\mathcal{G}_t-v_{tt} \right)+A_1*\left(f_t+\mathcal{G}_t-v_{tt} \right)   + A_2*\left(f+\mathcal{G}-v_{t} \right)\right] 
\eeq

where $A_1, A_2\in L^1_{[0,+\infty)}(\mathbb{R})$ are two functions that depend on $a_n$, 
with bounded $L^1$ norms which are independent of $n$, due to $(a_2)$ and $(a_5)$.  

We have the following estimate:

\begin{lemma}\label{sl}
Under the assumption stated in \eqref{pr6} one has

\begin{align}\label{sl1} 
& \int_\Omega v^2_{xx}(x,t)\ud x+ \int_0^t \int_\Omega  v^2_{xx}(x,s)\ud x\ud s+\int_0^t \int_\Omega  v^2_{tt}(x,s)\ud x\ud s \nonumber\\
& \leq C \left[ F + \mathcal{E}_1(t) + \nu(t) \mathcal{E}(t) \right]
\end{align}


\end{lemma}

\begin{proof}

\textbf{Step 1.}

We multiply \eqref{sl0} by $v_{xx}$ and integrate on $\Omega$.  It is clear that, for any $\eta>0$, we have

\beq\label{sl2} 
\left|\int_\Omega \left(f_t-v_{tt} \right)v_{xx}\ud x  \right|\leq \eta  \int_\Omega v^2_{xx} \ud x +\dfrac{1}{2\eta}\int_\Omega \left(f^2_t+v^2_{tt} \right)\ud x
\eeq

From part (iii) in Lemma \ref{le} we obtain

\begin{align}\label{sl3}
\left|\int_\Omega \mathcal{G}_t v_{xx} \ud x  \right| & \leq K \nu(t) \int_\Omega  
\left|v_{xx}(x,t)\right|\left(\left|v_{xx}\right|*|\psi| \right)(x,t)\ud x  \nonumber\\
& + \overline{a} K \nu(t) \int_\Omega \left|v_{xx}(x,t)\right|^2\ud x
\end{align}

Further, with the help of part (ii) in Lemma \ref{adl2} we obtain

\begin{align}\label{sl4}
\left|\int_\Omega \mathcal{G}_t v_{xx} \ud x  \right| & \leq K \nu(t) \left\|v_{xx}(\cdot,t) 
\right\|_{L^2(\Omega)} \|\psi\|_{L^1\left(\mathbb{R}_+ \right) } \displaystyle 
\mathop{\sup}_{0\leq \tau \leq t}\left\|v_{xx}(\cdot,\tau)\right\|_{L^2(\Omega)} \nonumber\\
& + \overline{a} K \nu(t) \left\|v_{xx}(\cdot,t) \right\|^2_{L^2(\Omega)} 
\leq K \nu(t) \left[ \|\psi\|_{L^1\left(\mathbb{R}_+ \right) }+\overline{a}\right]\mathcal{E}(t) 
\end{align}

For any $\eta>0$ one has

\begin{align}\label{sl5}
& \left|\int_\Omega A_1*\left(f_t-v_{tt} \right)v_{xx} \ud x  \right|\leq  \left\|A_1\right\|_{L^1\left(\mathbb{R}_+ \right) }  \left\|v_{xx}(\cdot,t) \right\|_{L^2\left(\Omega \right) } \displaystyle \mathop{\sup}_{0\leq \tau \leq t}\left[\left\|f_t(\cdot,\tau) \right\|_{L^2\left(\Omega \right) }+\left\|v_{tt}(\cdot,\tau) \right\|_{L^2\left(\Omega \right) } \right] \nonumber\\
& \leq \eta \left\|v_{xx}(\cdot,t) \right\|^2_{L^2\left(\Omega \right) }+\dfrac{1}{2\eta}\left\|A_1\right\|^2_{L^1\left(\mathbb{R}_+ \right) }\displaystyle \mathop{\sup}_{0\leq \tau \leq t}\left[\left\|f_t(\cdot,\tau) \right\|^2_{L^2\left(\Omega \right) }  + \left\|v_{tt}(\cdot,\tau) \right\|^2_{L^2\left(\Omega \right) }\right] 
\end{align}

and also

\begin{align}\label{sl6}
& \left|\int_\Omega A_2*\left(f-v_{t} \right)v_{xx} \ud x  \right| \nonumber\\
& \leq \eta \left\|v_{xx}(\cdot,t) \right\|^2_{L^2\left(\Omega \right) }+\dfrac{1}{2\eta} \left\|A_2\right\|^2_{L^1\left(\mathbb{R}_+ \right) } \displaystyle \mathop{\sup}_{0\leq \tau \leq t}\left[\left\|f(\cdot,\tau) \right\|^2_{L^2\left(\Omega \right) }+\left\|v_{t}(\cdot,\tau) \right\|^2_{L^2\left(\Omega \right) } \right]
\end{align}

We now have:

\begin{align}\label{sl7}
\left|\int_\Omega \left( A_1*\mathcal{G}_t\right)(x,t)v_{xx}(x,t)  \ud x  \right| & \leq 
\overline{a} K \nu(t) \int_\Omega \left( \left|A_1\right|*\left|v_{xx}\right| 
\right)(x,t)\left|v_{xx}(x,t)\right|  \ud x \nonumber\\
& +K \nu(t) \int_\Omega \left(\left|A_1\right|*\left|\psi\right|*\left|v_{xx}(x,t)\right|  
\right)(x,t) \left|v_{xx}(x,t)\right|  \ud x
\end{align}

Then:

\begin{align}\label{sl8}
& \left|\int_\Omega \left( A_1*\mathcal{G}_t\right)(x,t)v_{xx}(x,t)  \ud x  \right|  \nonumber\\
& \leq K \nu(t) 
\left[ 
\overline{a} \left\|A_1\right\|_{L^1\left(\mathbb{R}_+ \right) } +
\left\| \, |A_1| * |\psi| \, \right\|_{L^1\left(\mathbb{R}_+ \right) } 
\right]
\left\|v_{xx}(\cdot,t) \right\|_{L^2\left(\Omega \right) }
\mathop{\sup}_{0\leq \tau \leq t}\left\|v_{xx}(\cdot,\tau) 
\right\|_{L^2\left(\Omega \right) }.
\end{align}

This gives 

\beq\label{sl9}
\left|\int_\Omega \left( A_1*\mathcal{G}_t\right)(x,t)v_{xx}(x,t)  \ud x  \right|\leq 
C \nu(t) \mathcal{E}(t)
\eeq

Likewise,

\beq\label{sl10}
\left|\int_\Omega \left( A_2*\mathcal{G}\right)(x,t)v_{xx}(x,t)  \ud x  \right|\leq 
C \nu(t) \mathcal{E}(t)
\eeq

Now, from the above estimates \eqref{sl2}, \eqref{sl4}, \eqref{sl5}, \eqref{sl6}, \eqref{sl9} and 
\eqref{sl10}, with $\eta>0$ small enough 
leads to

\beq\label{sl11}
\displaystyle \mathop{\sup}_{0\leq s \leq t}\int_\Omega v^2_{xx}(x,s)  \ud x \leq C
\left[ F + \mathcal{E}_1(t) + \nu(t) \mathcal{E}(t) \right]
\eeq

\textbf{Step 2.}

We multiply \eqref{sl0} by $v_{xx}$ and integrate on $(0,t)$ and on $\Omega$.  Proceeding as in \textbf{Step 1.}, using part (i) in Lemma \ref{adl2}, one gets  for any $\eta>0$ that

\begin{align}\label{sl13n}
& \int_{Q_t} \left[f_t + \mathcal{G}_t + A_1*f_t + A_2*(f-v_t)+A_1*\mathcal{G}_t +A_2*\mathcal{G}  \right] v_{xx}\ud x \ud s \nonumber\\
& \leq \eta \int_{Q_t} v^2_{xx}\ud x \ud s + \dfrac{C}{\eta} \left[F+\mathcal{E}_1(t) \right]+C \nu(t) \mathcal{E}(t) 
\end{align}

We are left to focus on terms that contain $v_{tt}$.  Invoking density arguments,

\begin{align}\label{sl19}
\int_{Q_t} \left( v_{tt}v_{xx}\right) (x,s)\ud x \ud s = \int_\Omega  \left( v_{xx} v_t\right) (x,t)\ud x - \int_\Omega v_0''(x)v_t(x,0)\ud x+ \int_{Q_t} v^2_{xt}\ud x \ud s
\end{align}

which gives, using \eqref{eel9}, 

\begin{align}\label{sl20}
\left| \int_{Q_t} \left( v_{tt}v_{xx}\right) (x,s)\ud x \ud s \right| & \leq \left\|v_{xx}(\cdot,t) \right\|_{L^2\left(\Omega \right) } \left\|v_{t}(\cdot,t) \right\|_{L^2\left(\Omega \right) } \nonumber\\ 
& + \left\| v_0'' \right\|_{L^2\left(\Omega \right) } \left\|f(\cdot,0)\right\|_{L^2\left(\Omega \right) }+ \int_{Q_t}v^2_{xt}(x,s)\ud x \ud s 
\end{align}

Finally we have:

\begin{align}\label{sl21}
\int_{Q_t} \left( A_1*v_{tt} \right) (x,s)v_{xx}(x,s)\ud x \ud s & = \int_{Q_t} \left( A_1*v_{t} \right)_t v_{xx}(x,s)\ud x \ud s \nonumber\\ 
& - \int_{Q_t} A_1(s) v_t(x,0)v_{xx}(x,s)\ud x \ud s
\end{align}

Again, calling in the density arguments leads to

\begin{align}\label{sl22}
\int_{Q_t} \left( A_1*v_{t} \right)_t (x,s)v_{xx}(x,s)\ud x \ud s & = \int_\Omega 
\left( A_1*v_{t} \right)(x,t)v_{xx}(x,t)\ud x \nonumber\\
& + \int_{Q_t} \left( A_1*v_{xt} \right)v_{xt} \ud x \ud s
\end{align}

From equalities \eqref{sl21} and \eqref{sl22} one easily gets:

\begin{align}\label{sl23}
& \left| \int_{Q_t} \left( A_1*v_{tt}\right)v_{xx} (x,s)\ud x \ud s \right|\leq \|A_1\|_{L^1\left(\mathbb{R}_+ \right)} \nonumber\\  
& \left[\int_{Q_t}v^2_{xt}\ud x \ud s+ \left\|v_{xx}(\cdot,t) \right\|_{L^2\left(\Omega \right) }\displaystyle \mathop{\sup}_{0\leq \tau \leq t}  \left\|v_{t}(\cdot,t) \right\|_{L^2\left(\Omega \right) }+\left\|f(\cdot,0) \right\|_{L^2\left(\Omega \right) }\mathop{\sup}_{0\leq s \leq t}  \left\|v_{xx}(\cdot,s) \right\|_{L^2\left(\Omega \right) }  \right] 
\end{align}

Now, adding inequalities \eqref{sl13n}, \eqref{sl20}, \eqref{sl23} and upon using 
\eqref{sl11} it allows us to get

\beq\label{sl24}
\int_{Q_t} v^2_{xx}(x,t)  \ud x \leq C 
\left[ F + \mathcal{E}_1(t) + \nu(t) \mathcal{E}(t) \right]
\eeq



\textbf{Step 3.}

We now multiply \eqref{eel5} by $v_{tt}$ and integrate on $Q_t$.  We have the listed below results:

\beq\label{sl25}
\left| \int_{Q_t} v_{xx}v_{tt} \ud x \ud s \right| \leq \eta \int_{Q_t} v^2_{tt}\ud x \ud s + \dfrac{1}{4\eta} \int_{Q_t} v^2_{xx}\ud x \ud s 
\eeq

\begin{align}\label{sl26}
\int_{Q_t} \left(a'*v_{xx}\right) v_{tt} \ud x\ud s & \leq \|a'\|_{L^1\left(\mathbb{R}_+ \right)}\left\|v_{xx}  \right\|_{L^2\left(Q_t \right) }\left\|v_{tt}  \right\|_{L^2\left(Q_t \right) } \nonumber\\
& \leq \eta \left\|v_{tt}  \right\|^2_{L^2\left(Q_t \right) } +\dfrac{1}{4\eta}\|a'\|^2_{L^1\left(\mathbb{R}_+ \right)} \left\|v_{xx}  \right\|^2_{L^2\left(Q_t \right) }  
\end{align}

\beq\label{sl27}
\int_{Q_t}  f_t v_{tt}\ud x\ud s  \leq \eta \left\|v_{tt}  \right\|^2_{L^2\left(Q_t \right) } +\dfrac{1}{4\eta}\left\|f_{t}  \right\|^2_{L^2\left(Q_t \right) }
\eeq

\begin{align}\label{sl28}
\int_{Q_t}  \mathcal{G}_t  v_{tt} \ud x \ud s & \leq \overline{a} k \nu(t) \int_{Q_t} \left|v_{xx}  \right|\left|v_{tt}  \right| \ud x \ud s + k \nu(t) \int_{Q_t} \left(\left|v_{xx}  \right|*|\psi| \right) \left|v_{tt}  \right|\ud x \ud s \nonumber\\
& \leq k \nu(t) \left(\overline{a}+\|\psi\|_{L^1\left(\mathbb{R}_+ \right)} \right) \mathcal{E}(t)
\end{align}

We then obtain, taking $\eta$ small enough and using \eqref{sl24}, that

\beq\label{sl29}
\int_{Q_t}  v^2_{xx}(x,t)  \ud x \ud s \leq C 
\left[ F + \mathcal{E}_1(t) + \nu(t) \mathcal{E}(t) \right]
\eeq



Now from estimates \eqref{sl11}, \eqref{sl24} and \eqref{sl29} we obtain the result
of Lemma \ref{sl}.

\end{proof}

Now we take on to obtaining estimates for $u$ defined as $u(x,t)=\displaystyle\int_0^t v(x,s)\ud s$.  The idea is to integrate \eqref{sl0} w.r.t. $t$; one gets:

\begin{align}\label{sl30}
& u_{xx} = \nonumber\\
& \dfrac{1}{g'(0)} \left\{ \dfrac{f+\mathcal{G}-v_t}{a(0)}  + \int_0^t \left[A_1*\left(f_t+\mathcal{G}_t -v_{tt}\right)  \right] (x,s)  \ud s + \int_0^t \left[A_2*\left(f+\mathcal{G} -v_{t}\right)  \right] (x,s) \ud s \right\}
\end{align}

We shall use in the following the below Lemma:

\begin{lemma}\label{al}
Suppose that $A\in L^1\left( 0,T\right) $, $\f\in W^{1,1}\left( 0,T\right) $.  Then, for any $t\in (0,T)$, we have

\beq\label{al1}
\int_0^t (A*\f')(s)\ud s = A*\left[ \f - \f(0)H \right]  
\eeq
\end{lemma}

\begin{proof}
The proof is a direct consequence of Fubini's Theorem.

\end{proof}

Recall from \eqref{eel9} that $\left( f+\mathcal{G}-v_t\right) (x,0)=0$.  Then \eqref{sl30} can be re-written in the form

\begin{align}\label{al3}
& u_{xx}  = \nonumber\\
& \dfrac{1}{g'(0)} \left\{ \dfrac{f+\mathcal{G}-v_t}{a(0)} + A_1*\left(f+\mathcal{G} -v_{t}\right)
+ A_2*\left[\int_0^t f(x,s) \ud s + \int_0^t \mathcal{G}(x,s) \ud s-v+v_0 \right] \right\rbrace 
\end{align}

We deduce from the above equation that

\begin{align}\label{al3n}
& u_{xxx} =\dfrac{1}{g'(0)} \nonumber\\
&  \left\{\dfrac{f_x+\mathcal{G}_x-v_{xt}}{a(0)}   +A_1*\left(f_x+\mathcal{G}_x -v_{xt}\right)
+ A_2*\left[\int_0^t f_x(x,s) \ud s + \int_0^t \mathcal{G}_x(x,s) \ud s-v_x+v'_0 \right] \right\rbrace 
\end{align}

We can now prove the following:

\begin{lemma}\label{al4}
Assume the assumption formulated in \eqref{pr6} holds true.  Then

\begin{align}\label{al5}
& \displaystyle\mathop{\sup}_{0\leq s \leq t}\left\|u_{xx}(\cdot,s)\right\|^2_{L^2\left(\Omega \right) } \leq C \left\lbrace V_0+F+ \nu^2(t) \mathcal{E}(t) +\mathcal{E}^3(t)+\mathcal{E}_1(t)\right\rbrace     
\end{align}

and

\begin{align}\label{al5n}
& \displaystyle\mathop{\sup}_{0\leq s \leq t}\left\|u_{xxx}(\cdot,s)\right\|^2_{L^2\left(\Omega \right) } \leq C \left\lbrace V_0+F+ \nu^2(t) \mathcal{E}(t)+ \nu^2(t) \mathcal{E}^2(t) +\mathcal{E}^3(t)+\mathcal{E}_1(t)\right\rbrace     
\end{align}

where $C>0$ is a constant which is independent of $n$. 
\end{lemma}

\begin{proof}
The proof is performed in two steps.

\textbf{Step 1.}

Here we obtain the necessary estimates for $\mathcal{G}(t)$, $\displaystyle\int_0^t \mathcal{G}(s)\ud s$, $\mathcal{G}_x(t)$ and for $\displaystyle\int_0^t \mathcal{G}_x(s)\ud s$.  Using \eqref{p6} and part (i) of Lemma \ref{le} we have

\beq \label{al6}
\left|\mathcal{G}(t)\right|\leq K \nu(t) \int_0^{+\infty} \left|a'(s)\right| r_0(s) \left|u_{xx}(x,t)-u_{xx}(x,t-s)\right|\ud s
\eeq

and this gives

\beq \label{al7}
\left\|\mathcal{G}(\cdot,t)\right\|_{L^2\left(\Omega \right) }\leq 2 K \nu(t) \int_0^{+\infty} \left|a'(s)\right| r_0(s) \ud s \left( \displaystyle\mathop{\sup}_{0\leq s \leq t} \left\|u_{xx}(\cdot,s)\right\|_{L^2\left(\Omega \right) }\right) \leq C \nu(t) \sqrt{\mathcal{E}(t)}
\eeq

On the other hand, using \eqref{p6} and \eqref{ap7}, we have that

\beq \label{al8}
\left|\int_0^t \mathcal{G}(x,s)\ud s \right|\leq K \int_0^{+\infty} \left|a'(\tau)\right|\int_0^t \left|\overline{v}^s_x(x,\tau)\right|^2 \left|u_{xx}(x,s)-u_{xx}(x,s-\tau)\right|\ud s \ud \tau 
\eeq 

which implies, taking the ${L^2\left(\Omega \right) }$-norm, that

\begin{align}\label{al9}
\left\|\int_0^t \mathcal{G}(\cdot,s)\ud s \right\|_{L^2\left(\Omega \right) }  \leq & 2K \left( \displaystyle\mathop{\sup}_{0\leq \tau \leq t} \left\|u_{xx}(\cdot,\tau)\right\|_{L^2\left(\Omega \right) }\right)\nonumber\\
& \int_0^{+\infty} \left|a'(\tau)\right| \int_0^t \left\| \overline{v}^s_x(\cdot,\tau)\right\|^2_{L^\infty\left(\Omega\right) }  \ud s  \ud\tau
\end{align}

Now we have by Sobolev inclusions:

\beq \label{al10}
\left\|\overline{v}^s_x(\cdot,\tau)\right\|_{L^\infty\left(\Omega\right) } \leq C \int_{s-\tau}^s \left\| v(\cdot,\lambda)\right\|_{H^2\left(\Omega\right) }\ud\lambda \leq 2C   \tau \mathcal{M} \left( \left\| \tilde{v}\right\|_{H^2\left(\Omega\right) } \right) (s)
\eeq

where $ \tilde{v} (x,s)$ is the function defined on $\Omega\times \mathbb{R}$ by 

\begin{equation}\label{al11}
\tilde{v} (x,s)=\begin{cases}
v(x,s) & \textrm{for $s\in[0,t)$}\\
0 & \textrm{for $s\in\mathbb{R}-[0,t)$}                   
                  \end{cases}
\end{equation}

and 

\beq\label{al12}
\mathcal{M}\left(\left\|\tilde{v}\right\|_{H^2\left(\Omega\right) } \right) (s)=\displaystyle\mathop{\sup}_{\rho>0}\dfrac{1}{2\rho}\int_{s-\rho}^{s+\rho} \left\|\tilde{v}(\cdot,\tau)\right\|_{H^2\left(\Omega\right) } \ud\tau 
\eeq

is the maximal function of $s\mapsto\displaystyle\left\|\tilde{v}(\cdot,s)\right\|_{H^2\left(\Omega\right) }$ (see \cite{ste1}).  

Now, the maximal inequality (see Theorem 1, page 5 in \cite{ste1}) in this case leads  to 

\begin{align}\label{al13}
\int_{\mathbb{R}}  \mathcal{M}\left( \left\|\tilde{v}(\cdot,s)\right\|^2_{H^2\left(\Omega\right) } \right)(s)\ud s & \leq 2\sqrt{10}   \int_{\mathbb{R}}  \left\|\tilde{v}(\cdot,s)\right\|^2_{H^2\left(\Omega\right) } (x,s)\ud s \nonumber\\
& = 2\sqrt{10} \int_0^t \left\| v (\cdot,s)\right\|^2_{H^2\left(\Omega\right) } (x,s)\ud s
\end{align}


Then,  from \eqref{al10} and \eqref{al13} by Sobolev inclusions we have that:

\beq\label{al14}
\int_0^t \left\|\overline{v}_x^s(\cdot,\tau)\right\|_{L^\infty\left(\Omega\right)} ^2\ud\tau  \leq C \tau^2 \int_0^t \|v(\cdot,s)\|^2_{H^2(\Omega)}
\eeq

Next, with the help of \eqref{al9} we deduce

\begin{align}\label{al15}
\left\|\int_0^t \mathcal{G}(\cdot,s)\ud s \right\|_{L^2\left(\Omega \right) } \leq C K \displaystyle\mathop{\sup}_{0\leq \tau \leq t} \left\|u_{xx}(\cdot,\tau)\right\|_{L^2\left(\Omega \right) } \int_0^t \left\|v(\cdot,s)\right\|^2_{H^2(\Omega)}\ud s \int_0^{+\infty} \left|a'(\tau)\right| \tau^2  \ud\tau
\end{align}

that is 

\beq\label{al15n}
\left\| \int_0^t \mathcal{G}(\cdot,s)\ud s \right\|_{L^2\left(\Omega \right) } \leq C \mathcal{E}^{3/2}(t) 
\eeq

Next, let $\mathcal{G}_x(x,t)=I_1+I_2$, where

\beq\label{nal1}
I_1=\int_0^{+\infty}a'(s) g''\left(\overline{v}^t_x(s) \right)\left|\overline{v}^t_{xx}(s) \right|^2\ud s   
\eeq

\beq\label{nal2}
I_2=\int_0^{+\infty}a'(s)\left[g'\left(\overline{v}^t_x(s) \right)-g'(0) \right] \overline{v}^t_{xxx}(s)\ud s   
\eeq

and also $\displaystyle \int_0^t \mathcal{G}_x(x,s)\ud s =I_3+I_4$, where

\beq\label{nal3}
I_3= \int_0^t \int_0^{+\infty}a'(\tau) g''\left(\overline{v}^s_x(\tau) \right)\left|\overline{v}^s_{xx}(\tau) \right|^2\ud\tau \ud s 
\eeq

\beq\label{nal4}
I_4=\int_0^t \int_0^{+\infty}a'(\tau)\left[g'\left(\overline{v}^s_x(\tau) \right)-g'(0) \right] \overline{v}^s_{xxx}(\tau)\ud\tau\ud s   
\eeq

Since $\overline{v}^t(s)=u(t)-u(t-s)$, using again part (i) in Lemma \ref{le} we obtain

\begin{align}\label{nal5}
\left\|I_1\right\|_{L^2\left(\Omega \right) } & \leq 2K \nu(t)\int_0^{+\infty} \left|a'(s)\right|r_0(s)\left[\left\|u^2_{xx}(\cdot,t) \right\|_{L^2\left(\Omega \right) }+\left\|u^2_{xx}(\cdot,t-s) \right\|_{L^2\left(\Omega \right) }  \right] \ud s\nonumber\\
& \leq  4K \nu(t) \displaystyle\mathop{\sup}_{0\leq s \leq t} \left\|u_{xx}(\cdot,s)\right\|^2_{L^4\left(\Omega \right) }\int_0^{+\infty}\left|a'(s)\right|r_0(s)\ud s 
\end{align}

This gives further down by Sobolev inclusion:

\begin{align}\label{nal6}
\left\|I_1\right\|_{L^2\left(\Omega \right) } & \leq 4 K \left( \int_0^{+\infty} \left|a'(s)\right|r_0(s)\ud s\right) \nu(t) \mathcal{E}(t) 
\end{align}

Next, as in \eqref{al7}, one easily obtains that

\begin{align}\label{nal7}
\left\|I_2\right\|_{L^2\left(\Omega \right) } & \leq 2 K \left( \int_0^{+\infty} \left|a'(s)\right|r_0(s)\ud s\right) \nu(t) \sqrt{\mathcal{E}(t)} 
\end{align}

Moreover,

\begin{align}\label{nal9}
& \left\|I_3\right\|_{L^2\left(\Omega \right) } \leq \nonumber\\
&  K \int_0^t \int_0^{+\infty} \left|a'(\tau)\right| \left\|\overline{v}^s_{x}(\cdot,\tau) \right\|_{L^\infty\left(\Omega \right) }  \left\|u_{xx}(\cdot,s)-u_{xx}(\cdot,s-\tau) \right\|_{L^\infty\left(\Omega \right) }\left\|  \overline{v}^s_{xx}(\cdot,\tau)\right\|_{L^2\left(\Omega \right) }    \ud\tau\ud s
\end{align}

As in the proof of \eqref{al5} we have the following estimates:

\[ \left\|\overline{v}^s_{x}( \tau) \right\|_{L^\infty\left(\Omega \right) }\leq  2 \tau \mathcal{M} \left( \left\|\tilde{v}_{x}  \right\|_{L^\infty\left(\Omega \right) }\right) (s) \]

\[ \left\|\overline{v}^s_{xx}( \tau) \right\|_{L^2\left(\Omega \right) }\leq  2 \tau \mathcal{M}\left(  \left\|\tilde{v}_{xx}  \right\|_{L^2\left(\Omega \right) }\right) (s) \]

which give

\begin{align}\label{nal10}
\left\|I_3\right\|_{L^2\left(\Omega \right) } & \leq 8K \displaystyle\mathop{\sup}_{0\leq s \leq t} \left\|u_{xx}(\cdot,s)\right\|_{L^\infty\left(\Omega \right) } \int_0^{+\infty}  \left|a'(\tau)\right|\tau^2\ud\tau\nonumber\\
& \sqrt{\int_0^t \mathcal{M}   \left( \left\|\tilde{v}_{x}\right\|_{L^\infty\left(\Omega \right) }\right)^2 (s)\ud s} \sqrt{\int_0^t \mathcal{M}   \left( \left\|\tilde{v}_{xx}\right\|_{L^2\left(\Omega \right) }\right)^2(s)\ud s}
\end{align}

Using again the maximal inequality from \cite{ste1} and the Sobolev embeddings leads to

\begin{align}\label{nal11}
& \left\|I_3\right\|_{L^2\left(\Omega \right) } \leq C  \displaystyle\mathop{\sup}_{0\leq s \leq t} \left\|u (\cdot,s)\right\|_{H^3\left(\Omega \right) }\int_0^t \left\|v (\cdot,s)\right\|^2_{H^2\left(\Omega \right) }\ud s
\end{align}

that is

\begin{align}\label{nal12}
\left\|I_3\right\|_{L^2\left(\Omega \right) } \leq C \mathcal{E}^{3/2}(t)
\end{align}

Finally, for $I_4$ we proceed as for obtaining \eqref{al15n} and get 

\begin{align}\label{nal13}
\left\|I_4\right\|_{L^2\left(\Omega \right) } \leq C \mathcal{E}^{3/2}(t) 
\end{align}

The above estimates lead to the below ones:

\begin{align}\label{nal14}
\left\|\mathcal{G}_x(\cdot,t)\right\|_{L^2\left(\Omega \right) } \leq C \nu(t)\left(  \mathcal{E}(t)+\sqrt{\mathcal{E}(t)}\right) 
\end{align}

\begin{align}\label{nal15}
\left\|\int_0^t \mathcal{G}_x(\cdot,s)\ud s\right\|_{L^2\left(\Omega \right) } \leq C  \mathcal{E}^{3/2}(t) 
\end{align}

\textbf{Step 2.}

From \eqref{al3} we obtain:

\begin{align}\label{al17}
& \left\|u_{xx}(\cdot,t)\right\|_{L^2\left(\Omega \right) } \leq  \dfrac{1}{|g'(0)|} \bigg\{ \dfrac{1}{a(0)}\left[\|f(\cdot,t)\|_{L^2\left(\Omega \right) } + \left\|\mathcal{G}(\cdot,t)\right\|_{L^2\left(\Omega \right) } +\left\|v_t(\cdot,t)\right\|_{L^2\left(\Omega \right) }  \right] \nonumber\\
& + \|A_1\|_{L^1\left(\mathbb{R}_+ \right)} \displaystyle\mathop{\sup}_{0\leq s \leq t}\left[ \|f(\cdot,s)\|_{L^2\left(\Omega \right) } + \left\|\mathcal{G}(\cdot,s)\right\|_{L^2\left(\Omega \right) } + \left\|v_t(\cdot,s)\right\|_{L^2\left(\Omega \right) }\right] \nonumber\\
& + \|A_2\|_{L^1\left(\mathbb{R}_+ \right)} \displaystyle\mathop{\sup}_{0\leq s \leq t}\left[ \left\| \int_0^s f(\cdot,\tau)\ud\tau\right\|_{L^2\left(\Omega \right) } + \left\|\int_0^s \mathcal{G}(\cdot,\tau)\ud\tau\right\|_{L^2\left(\Omega \right) } + \left\|v(\cdot,s)\right\|_{L^2\left(\Omega \right) } +\left\|v_0\right\|_{L^2\left(\Omega \right) } \right] \bigg\}
\end{align}

Using now \eqref{al7} and \eqref{al15n} and the fact that $\nu(t)$ and $\mathcal{E}(t)$ are increasing functions we obtain \eqref{al5}.  Next,\eqref{al5n} is obtained in a similar manner: one produces an equality  like that of \eqref{al17} satisfied by $\displaystyle\left\|u_{xxx}(\cdot,t)\right\|_{L^2\left(\Omega \right) }$ with $f_x$, $\mathcal{G}_x$, $v_{tx}$, $v_x$, $v'_0$ in place of $f$, $\mathcal{G}$, $v_{t}$, $v$, $v_0$.  Using \eqref{nal14} and \eqref{nal15} we get \eqref{al5n}.  

\end{proof}

\subsection{Smallness estimates.}

The next Proposition proves the uniform boundedness of $\mathcal{E}(t)$.

\begin{proposition}\label{smr}
There exist two numbers $\overline{\mathcal{E}}>0$ and $\delta>0$ independent of $n$ such that, whenever $v_0$ and $f$ verify $F(f)+V_0(v_0)\leq \delta$, one has

\beq\label{smr1}
\mathcal{E}(t) \leq \dfrac{\overline{\mathcal{E}}}{2},\,\forall t\in[0,T)
\eeq 
\end{proposition}

\begin{proof}
Remark first that, capitalizing on \eqref{eel5} and \eqref{pr5}, one has $v_t(x,0)=f(x,0)$, $v_{xt}(x,0)=f_x(x,0)$,   $v_{tt}(x,0)=-g'(0)a(0)v''_0(x)+f_t(x,0)$.  From the definition of $\mathcal{E}(t)$ we deduce 

\begin{align}\label{smr2}
\mathcal{E}(0)  \leq \left[1+2a^2(0)\left|g'(0)\right|^2  \right] \left\|v_{0}\right\|^2_{H^2\left(\Omega \right) }
+ \int_\Omega \left[f^2(x,0)+f^2_x(x,0)+2f^2_t(x,0) \right] \ud x
\end{align}

Therefore

\beq\label{smr3}
\mathcal{E}(0)  \leq 2 \left[1+a^2(0) \left|g'(0)\right|^2\right]\left(F+V_0 \right)   
\eeq

We now use the fact that the seminorm $w\in H^2\left(\Omega \right) \mapsto \left\|w_{xx}\right\|_{L^2\left(\Omega \right) }$ is a norm on $H^2\left(\Omega \right)\cap H^1_0\left(\Omega \right)$, equivalent to the usual norm in $H^2\left(\Omega \right)$.  We shall as well make use of the inequality $\left(\sqrt{V_0}+\sqrt{F} \right)  \sqrt{\mathcal{E}(t)}\leq \eta\mathcal{E}(t)+\dfrac{1}{2\eta}(V_0+F)$, with $\eta>0$ small enough. 

From Lemmas \ref{lse}, \ref{sl} and \ref{al4} we deduce 

\beq \label{smr4}
\mathcal{E}(t)  \leq C \left\lbrace V_0+F+\left[\nu(t)+\nu^3(t) \right]\mathcal{E}(t)+\sqrt{V_0}\mathcal{E}(t)+\mathcal{E}^3(t)+\nu^2(t)\mathcal{E}^2(t)  \right\rbrace 
\eeq

provided \eqref{pr6} holds true.  



Recall also the inequality \eqref{pr4}:

\beq\label{smr7}
\nu(t)\leq c_\Omega \sqrt{\mathcal{E}(t)},\, \forall  t\in [0,T)
\eeq 

Then, we deduce from \eqref{smr4} that  



\beq \label{smr4n}
\mathcal{E}(t)  \leq c_1 \left[ V_0+F+\mathcal{E}^3(t)\right]
\eeq

with $c_1>0$ a constant independent of $n$.  

Now observe that we can choose $\overline{\mathcal{E}}>0$ and $\delta>0$ such that

\beq\label{smr8}
\begin{cases}
c_1   \overline{\mathcal{E}}^{2}  \leq \dfrac{1}{2} \\*[2ex]
\overline{\mathcal{E}}<\dfrac{\theta^2}{4 C_\Omega} \\*[2ex]
c_1 \delta \leq \dfrac{\overline{\mathcal{E}}}{4}\\*[2ex] 
2\left[1+a^2(0)\left|g'(0)\right|^2  \right]\delta \leq \dfrac{\overline{\mathcal{E} } }{2}\\ 
\end{cases}
\eeq

Let us now prove that, for any $t\in[0,T)$, \eqref{smr1} holds true.  Indeed, if the contrary were true, then invoking the continuity w.r.t. time there exists $t_2\in (0,T)$ s.t. $\mathcal{E}(t)\leq \overline{\mathcal{E} }$, for any $t\in(0,t_2)$, but  inequality  \eqref{smr1} is false on an interval $(t_1,t_2)$ with $0<t_1<t_2$.  From the second inequality in \eqref{smr8} we deduce that \eqref{smr4n} is satisfied on $[0,t_2]$.  Using once more \eqref{smr8} one gets $\displaystyle \mathcal{E}(t)\leq \dfrac{\mathcal{E}(t)}{2}+\dfrac{\overline{\mathcal{E}}}{4}$ which triggers $\displaystyle \mathcal{E}(t)\leq \dfrac{\overline{\mathcal{E}}}{2}$ on $[0,t_2]$, hence a contradiction.  This later fact ends the proof.

\end{proof}

\section{Proof of the main result.}\label{pmrs}

Remark that from Proposition \ref{smr} we actually deduce that for $v_n$ -  solution of $(P_n)_1$, $(P_n)_2$, $(P_n)_3$   - we have the following upper bounds:

\begin{align}\label{ub1}
& \displaystyle \mathop{\sup}_{t\in[0,T_n)} \left[\|u_n(\cdot, t)\|^2_{H^3(\Omega)} +
\|(u_n)_t(\cdot, t)\|^2_{H^2(\Omega)} + \|(u_n)_{tt}(\cdot, t)\|^2_{H^1(\Omega)} +
\|(u_n)_{ttt}(\cdot, t)\|^2_{L^2(\Omega)} \right] \nonumber\\
& + \int_0^{T_n}\left\lbrace \|v_n(\cdot, t)\|^2_{H^2(\Omega)} + 
\|(v_n)_t (\cdot, t)\|^2_{H^1(\Omega)} + \|(v_n)_{tt} (\cdot, t)\|^2_{L^2(\Omega)} 
\right\rbrace \ud t \leq \dfrac{\overline{\mathcal{E}}}{2}
\end{align}


and 

\beq\label{ub2}
\sup_{\substack{x\in\Omega\\0<s<t<T_n}}\left|\int_{t-s}^t \left( v_n\right)_x(x,\tau)\ud \tau \right|\leq \theta
\eeq

We then deduce from Proposition \ref{pr} that $T_n=+\infty$, so \eqref{ub1} and \eqref{ub2} are 
valid upon replacing $T_n$ by $+\infty$.  It follows that there exist 
two limits 
$$
\displaystyle u \in \bigcap_{m = 0}^3 W^{m, \infty} \left( (0, + \infty) ; H^{3-m}(\Omega) \right)
   $$

and
 
$$
\displaystyle v \in \left\lbrace \bigcap_{m = 0}^2 W^{m, \infty} \left( (0, + \infty) ; H^{2-m}(\Omega) \right)\right\rbrace 
\cap \left\lbrace \bigcap_{m = 0}^2 W^{m, 2} \left( (0, + \infty) ; H^{2-m}(\Omega) \right)\right\rbrace 
   $$

with $\displaystyle u(x,t) = \int_0^t v(x,s) \, ds $ 
s.t. (up to a subsequence of $n$) we have 
$$
\frac {d^m u_n} {d t^m} \rightharpoonup \frac {d^m u} {d t^m} \quad \text{weakly} * \; \; \text{in}
\; \; L^\infty \left( (0, + \infty) ; H^{3-m}(\Omega) \right), \quad m = 0, 1, 2, 3
   $$
and
$$
\frac {d^m v_n} {d t^m} \rightharpoonup \frac {d^m v} {d t^m} \quad \text{weakly}  \; \; \text{in}
\; \; L^2 \left( (0, + \infty) ; H^{2-m}(\Omega) \right), \quad m = 0,1,2.
   $$
By the trace theorem we have $v=0$ for $x\in\partial\Omega$, $t\geq0$, and $v(x,0)=v_0(x)$, 
for $x\in\Omega$. 
Now remark that the equation $(P_n)_1$ can be written in the form


\begin{align}\label{ub3} 
(v_n)_t (x,t)=-\dfrac{\partial}{\partial x}\int_0^t a_n(t-s) g'\left((u_n)_x(x,t) - (u_n)_x(x,s)
\right) (v_n)_x(x,s)\ud s + f(x,t) 
\end{align}

We now pass to the limit in \eqref{ub3} above, for any fixed $t\geq0$.  By the trace theorem it is 
clear that  $\displaystyle (v_n)_t(\cdot,t) \xrightarrow[n\to+\infty]{L^2\left(\Omega \right)} 
v_t(\cdot,t)$ weakly.  Next, we take on to proving that 

$$\displaystyle \int_0^t a_n(t-s)g'\left( (u_n)_x(x,t) - (u_n)_x (x,s) \right) (v_n)_x(x,s)\ud s$$ 

weakly converges in $L^2\left(\Omega \right)$ towards  

$$\displaystyle \int_0^t a(t-s)g'\left(u_x(x,t) - u_x(x,s) \right) v_x(x,s)\ud s$$

Let $\phi\in L^2\left(\Omega \right)$ be fixed; we have to prove that

\beq\label{ub4}
E_n \displaystyle \xrightarrow[n\to+\infty]{\,}E
\eeq

where

\beq\label{ub5}  
E_n=\int_{Q_t} \phi(x) a_n(t-s)g'\left( (u_n)_x(x,t) - (u_n)_x(x,s) \right) (v_n)_x(x,s)\ud x \ud s
\eeq

\beq\label{ub6}  
E= \int_{Q_t} \phi(x) a(t-s)g'\left( u_x(x,t) - u_x(x,s) \right) v_x(x,s)\ud x \ud s
\eeq

By Sobolev compact inclusion we have that $\displaystyle (u_n)_x \xrightarrow[n\to+\infty]
{C\left(\overline{Q_t} \right)} u_x$ \ strongly  
\\
and 
$\displaystyle (u_n)_x(\cdot, t) \xrightarrow[n\to+\infty]
{C\left(\overline{\Omega} \right)} u_x(\cdot, t)$ \ also strongly. 
From \eqref{ub2}, with $T_n=+\infty$ we deduce

\beq\label{ub7}
\sup_{\substack{x\in\Omega\\0<s<t}}\left|\int_{t-s}^t v_x(x,\tau)\ud \tau \right|\leq \theta
\eeq


Making use of \eqref{ap6} leads to the strong convergence

\beq\label{ub8}
g'\left( (u_n)_x(x, t) - (u_n)_x(x, s) \right) \xrightarrow[n\to+\infty]
{C\left(\overline{Q_t} \right)} 
g'\left( u_x(x,t) - u_x(x,s) \right).
\eeq

Since $(v_n)_x\displaystyle \xrightarrow[n\to+\infty]{L^2(Q_t)} v_x$ strongly and $a_n \displaystyle \xrightarrow[n\to+\infty]{L^2(0,t)} a$ strongly (consequence of assumption $(a_2)$), one easily gets \eqref{ub4} which ends the proof of Theorem \ref{mr}.

\section{A class of totally monotone functions compliant with hypotheses $(a_1)$ to $(a_5)$.}\label{ax}

The goal here is to introduce a large class of functions $a$ compliant with assumptions $(a_1)$-$(a_5)$.  The following Lemma deals with sufficient conditions so that $(a_5)$ holds. 

\begin{lemma}\label{ax1}
Assume that $b\in W^{1,1}\left(0,+\infty \right)$ satisfies the following conditions

\begin{enumerate}[(i)]
\item $ t b'\in L^1\left(0,+\infty \right)$
\item there exists $M_3>0$ and $\alpha_1>0$ s.t. $\left|\mathcal{F}b(\omega) \right|\geq \dfrac{M_3}{1+|\omega|^{\alpha_1}} $, $\forall \omega\in\mathbb{R}$
\item there exists $M_4>0$ and $\alpha_2>0$ s.t. $\left| \mathcal{F}b'(\omega) \right|\leq \dfrac{M_4}{1+|\omega|^{\alpha_2}} $, $\forall \omega\in\mathbb{R}$
\item there exists $\alpha_3\in \mathbb{R}$ s.t. the function $\mathbb{R}\ni t\mapsto t b(t)\in \mathbb{R}$ is an element of $H^{\alpha_3}\left(\mathbb{R} \right)$ 
\end{enumerate}

Then there exists $M_5>0$ depending only on $M_3$,$M_4$, $\alpha_1$, $\alpha_2$ and $\alpha_3$, and $p\in\mathbb{N}^*$ depending only on $\alpha_1$ and $\alpha_2$ and $\alpha_3$, s.t. 

\beq\label{ax2}
\dfrac{\left( \mathcal{F}b'\right)^p}{\mathcal{F}b}\in \mathcal{F}\left( B_{L^1\left(\mathbb{R} \right)} (0,M_6) \right) 
\eeq 

where 

\beq\label{ax3}
M_6=M_5\left[1+ \|tb'\|_{L^1\left(\mathbb{R} \right)}+\|tb\|_{H^{\alpha_3}\left(\mathbb{R} \right)}  \right] 
\eeq

\end{lemma}

\begin{proof}

Since $H^1\left(\mathbb{R} \right)\subset \mathcal{F}L^1\left(\mathbb{R} \right)$ and $\left\|\mathcal{F}^{-1}w \right\|_{L^1\left(\mathbb{R} \right)}\leq C \|w\|_{H^1\left(\mathbb{R} \right)}$, $\forall w\in  H^1\left(\mathbb{R} \right)$ (see \cite{limg}), it suffices to consider the $H^1$ norm of $E\equiv \dfrac{\left[ \mathcal{F}b'\right]^p}{\mathcal{F}b}$.  From hypotheses (ii) and (iii) it is clear that, for $p$ large enough depending on $\alpha_1$ and $\alpha_2$, we have

\beq\label{ax4}
\|E\|_{L^2\left(\mathbb{R} \right)}\leq M_5
\eeq

where $M_5$ depends on $M_3$, $M_4$ and $\alpha_2$.  We also have $E'=E_1-E_2$, with 

\beq\label{ax5}
E_1:= p \dfrac{\left[ \mathcal{F}b'\right]^{p-1} \left[ \mathcal{F}b'\right]'}{\mathcal{F}b}
\eeq

\beq\label{ax6}
E_2:= \dfrac{\left[ \mathcal{F}b'\right]^{p} \left[ \mathcal{F}b\right]'}{\left( \mathcal{F}b\right)^2}
\eeq

Since $\left| \left( \mathcal{F}b'\right)' \right| = 
\left|  \mathcal{F}(tb') \right| \in L^\infty\left(\mathbb{R}_+ \right)$, from the above mentioned assumptions we get there exists $p$ large enough depending on $\alpha_1$ and $\alpha_2$ s.t.

\beq\label{ax7}
\|E_1\|_{L^2\left(\mathbb{R} \right)}\leq M_5 \|tb'\|_{L^1\left(\mathbb{R} \right)}
\eeq

From assumption (iv) and the fact that  $\left| \left( \mathcal{F}b\right)' \right| =
\left| \mathcal{F}(tb) \right| $  
we have that the function $\omega \longrightarrow \left(1+\omega^2 \right)^{\alpha_3/2}\left( \mathcal{F}b\right)'(\omega)\in L^2(\mathbb{R})$, and, $\left\|\left(1+\omega^2 \right)^{\alpha_3/2}\left( \mathcal{F}b\right)'(\omega)  \right\|_{L^2(\mathbb{R})} = \|tb\|_{H^{\alpha_3}(\mathbb{R})}$.

Then there exists $p$ large enough depending on $\alpha_1$,  $\alpha_2$ and $\alpha_3$ s.t.     

\beq\label{ax8}
\|E_2\|_{L^2\left(\mathbb{R} \right)}\leq M_5 \|tb\|_{H^{\alpha_3}\left(\mathbb{R} \right)}
\eeq

with $M_5$ as before.  From \eqref{ax4}, \eqref{ax7} and \eqref{ax8}  the claimed result follows. 

\end{proof}

Let $\mu$ be a positive, finite and non-zero Borel measure on $\mathbb{R}_+$, satisfying

\begin{enumerate}[$(\mu_1):$]
\item the function $\mathbb{R}_+\ni \rho \mapsto \dfrac{1}{\rho^2}$ is an element of $L^1_\mu(0,+\infty)$
\item there exists $\gamma\in(0,1)$ s.t. the function $\mathbb{R}_+\ni \rho \mapsto \rho^\gamma$ 
is an element of $L^1_\mu(0,+\infty)$
\end{enumerate}

\noindent
Remark that, as a consequence of these hypotheses, the function
$\mathbb{R}_+\ni \rho \mapsto \rho^\beta$ 
is an element of $L^1_\mu(0,+\infty)$ for any $\beta \in [-2, \gamma]$.

We now consider the following totally monotone function (see \cite{pru1}) 

\beq\label{ax8p}
\tilde{a}:[0,+\infty)\rightarrow \mathbb{R},\, \tilde{a}(t)=\int_{\mathbb{R}_+} e^{-\rho t} \ud \mu(\rho),\,\forall t\geq 0
\eeq

This Section main result is contained in the below theorem:

\begin{theorem}\label{axr}
Assume the hypotheses $(\mu_1)$ and $(\mu_2)$ hold true.  Then the function $\tilde{a}$ given by 
\eqref{ax8p} satisfies the hypotheses $(a_1)$-$(a_5)$ of Section \ref{ibvps} with 
\\

$$\tilde{a}_n(t)=\displaystyle \int_{[0,n)}e^{-\rho t}\ud \mu(\rho),\, \forall t\geq 0,\, \forall n\in \mathbb{N}^*$$
\end{theorem}

\begin{proof}

Since the measure $\mu$ is finite, it is clear that $\tilde{a}_n \in\mathscr{C}^\infty
\left(\mathbb{R}_+ \right)$, and for any $t\in  \mathbb{R}_+$ and $k\in \mathbb{N}$, 
$(\tilde{a}_n)^{(k)}(t)= \displaystyle \int_{[0,n)} (-1)^k \rho^k e^{-\rho t}\ud \mu(\rho)$.  
This gives $\tilde{a}_n \in W^{p,\infty}\left(0,+\infty \right) $, for any $p\in \mathbb{N}$
and also $\tilde{a}_n' < 0$. 

Let $k\in \mathbb{N}$ and $q\in \mathbb{R}_+$.  Then

$$\int_0^{+\infty} t^q \left(\tilde{a}_n \right)^{(k)}(t) \ud t = (-1)^k \int_0^{+\infty} 
t^q\int_{[0,n)} \rho^k e^{-\rho t}\ud \mu(\rho)\ud t = (-1)^k \int_{[0,n)} \rho^k \left( 
\int_0^{+\infty}t^q e^{-\rho t} \ud t \right)   \ud \mu(\rho) $$

Taking $\tau=\rho t$ in the integral w.r.t. $t$ leads to

\beq\label{axr1}
\int_0^{+\infty} t^q \left| \left(\tilde{a}_n \right)^{(k)}(t) \right| \ud t = \int_0^{+\infty} 
\tau^q e^{-\tau}\ud\tau  \int_{[0,n)} \rho^{k-q-1}\ud \mu(\rho)
\eeq

Invoking hypotheses $(\mu_1)$ and $(\mu_2)$ gives

\beq\label{axr2}
\int_{[0, + \infty)} \rho^{k-q-1}\ud \mu(\rho) <\infty
\eeq

provided that 

\beq\label{axr3}
0\leq q+1-k\leq 2
\eeq

For $q=0$ and $k=0$ or $k=1$ one sees that \eqref{axr3} is verified, therefore $(a_1)$ and $(a_2)$ are valid.

For $q=2$ and $k=1$ \eqref{axr3} is also verified, then $\displaystyle 
\int_0^{+\infty} t^2 |\tilde{a}'_n (t)| \ud t$ is bounded.  The same for $q=1$ and $k=2$, 
with this time $\displaystyle \int_0^{+\infty}t |\tilde{a}''_n (t)|\ud t$ bounded.  
The later grants $(a_3)$ is valid.

Next, by Fubini's theorem we obtain, for $\omega\in \mathbb{R}$,

$$\mathcal{F}\tilde{a}_n(\omega)=\int_0^{+\infty}\int_{[0,n)} e^{-\rho t}\ud \mu(\rho) e^{-i\omega t}\ud t=\int_{[0,n)} \dfrac{\ud \mu(\rho)}{\rho+i\omega}  $$

from which one gets

$$\text{Re}\left[\mathcal{F}\tilde{a}_n(\omega) \right]=\int_{[0,n)}\dfrac{\rho}{\rho^2+\omega^2}\ud \mu(\rho)    $$

Now, assumption $(\mu_1)$ gives $\mu \left(\left\lbrace 0 \right\rbrace \right)=0$, so, there exists $\underline{\mu}$  and $\overline{\mu}$ s.t. $0<\underline{\mu}<\overline{\mu}$ and $\mu\left(\left[\underline{\mu}, \overline{\mu} \right]  \right) >0$.  Take $n>\overline{\mu}$ to get 

$$\text{Re}\left[\mathcal{F}\tilde{a}_n(\omega) \right]\geq \dfrac{\underline{\mu}}{\overline{\mu}^2+\omega^2}\mu\left(\left[\underline{\mu}, \overline{\mu} \right]  \right),\, \forall \omega\in \mathbb{R}$$

which proves $(a_4)$. 

Now we prove that the hypotheses of Lemma \ref{ax1} are verified for $ b = \tilde{a}_n$, with constants independent of $n$.

The last inequality also proves that (ii) of Lemma \ref{ax1} is verified with  $M_3$
independent of $n$ and $\alpha_1=2$.  Taking $q = k = 1$ (which satisfy \eqref{axr3}) we deduce that part (i) of Lemma \ref{ax1} is also 
verified, and that $\|t \tilde{a}_n'\|_{L^1(0, + \infty)}$ is bounded.

Next, on one hand, we easily calculate 

$$\mathcal{F}\tilde{a}'_n(\omega)=-\int_{[0,n)} \dfrac{\rho}{\rho+i\omega}\ud \mu(\rho) $$

which gives

\beq\label{axr4} 
\left| \mathcal{F}\tilde{a}'_n(\omega) \right| \leq \int_{[0, n)} \dfrac{\rho}{\sqrt{\rho^2+\omega^2}} \ud \mu(\rho)
\eeq

We deduce that

\beq\label{axr5} 
\left| \mathcal{F}\tilde{a}'_n(\omega) \right| \leq \int_{\mathbb{R}_+} \ud \mu(\rho)
\eeq

On the other hand now, we use the fact that 

$$\rho^{2(1-\gamma)}|\omega|^{2\gamma}\leq   \gamma   |\omega|^2+ 
 (1-\gamma)  \rho^2 \leq |\omega|^2+\rho^2 $$

to get from \eqref{axr4}, for $\omega\neq0$,

$$\left| \mathcal{F}\tilde{a}'_n(\omega) \right| \leq \int_{[0,n)} \dfrac{\rho}{\rho^{1-\gamma} |\omega|^\gamma} \ud \mu(\rho)=\dfrac{1}{|\omega|^\gamma}\int_{[0,n)}\rho^{ \gamma} \ud \mu(\rho)  $$

Invoke $(\mu_2)$ to get, for $\omega\neq0$, 

\beq\label{axr6}
\left| \mathcal{F}\tilde{a}'_n(\omega) \right| \leq \dfrac{1}{|\omega|^\gamma}\int_{\mathbb{R}_+}\rho^{ \gamma} \ud \mu(\rho)
\eeq

Then, \eqref{axr5} and \eqref{axr6} give

$$\left| \mathcal{F}\tilde{a}'_n(\omega) \right| \leq \dfrac{2}{1+|\omega|^\gamma} \int_{\mathbb{R}_+} \left(1+ \rho^{ \gamma} \right)\ud \mu(\rho)   $$

Then the assumption formulated in (iii) of Lemma \ref{ax1} is verified
with $\alpha_2=\gamma$ and a constant $M_4$ independent of $n$.

Finally, the inequality \eqref{axr3} is verified with $q=1$ and $k=0$.  From \eqref{axr1} and assumption $(\mu_2)$ we get 

$$ \|t\tilde{a}_n\|_{L^1\left(\mathbb{R}_+ \right)} \leq \int_0^{+\infty} \tau e^{-\tau} \ud\tau  \int_{\mathbb{R}_+} \rho^{ -2} \ud \mu(\rho)< \infty $$

The above entails $t\tilde{a}_n$ is bounded in $H^{-1}(\mathbb{R})$; consequently hypothesis 
(iv) of Lemma \ref{ax1} is verified with $\beta=-1$.
We then deduce that the conclusion of Lemma \ref{ax1} is verified with a constant $M_6 > 0$
independent of $n$. 
Then hypothesis $(a_5)$ is verified.
\end{proof}

\begin{remark}\label{frk}
The relaxation  function of the Doi-Edwards theory, $\displaystyle a_{\text{DE}}(t)=\sum_{k=1}^{+\infty}\dfrac{1}{(2k+1)^2}e^{-(2k+1)^2 t}$, $t\geq0$, is actually a particular case of \eqref{ax8p} with the measure $\displaystyle \mu_{\text{DE}} =\sum_{k=1}^{+\infty}\dfrac{1}{(2k+1)^2}\delta_{(2k+1)^2}$, where $\delta_{(2k+1)^2}$ is Dirac's measure at $(2k+1)^2$. 

It is easy to see that the assumptions $(\mu_1)$, $(\mu_2)$ are verified for this measure, and this paper results can be applied for the $\displaystyle a_{DE}$  function.  
\end{remark}


\section{Appendix.}

The task here is to prove Lemma \ref{lies}, relabeled below as Lemma \ref{lies1x}.

Let the function $\xi=\xi(s,t,x)$ be defined a.e. as $\displaystyle \xi(s,t,x):=a'(s)\left[g'\left(\overline{v}_x^t(x,s)\right) - g'(0)   \right] $, $s\in[0,+\infty)$, $t\in[0,T)$, $x\in\Omega$.  Let $\displaystyle D_T:=\left\lbrace (s,t): \, s\in[0,+\infty),\, t\in[0,T),\,s\neq t \right\rbrace $.

In the following,  $\partial_1\xi$,  $\partial_2\xi$, $\partial_{22}\xi$ stand for $\dfrac{\partial \xi}{\partial s}$, $\dfrac{\partial \xi}{\partial t}$, and $\dfrac{\partial^2 \xi}{\partial t^2}$, respectively.


The first step is proving the following:

\begin{lemma}\label{ies}
Invoking the above defined notations, 
\begin{enumerate}[(i)]
\item one has: $\displaystyle \xi \in \mathscr{C}^1\left(D_T; H^1(\Omega) \right) $,  $\dfrac{\partial^2 \xi}{\partial t^2}\in \mathscr{C}^0\left(D_T; L^2(\Omega) \right) $;
\item assuming \eqref{pr6} holds true,  one has the following estimates a.e. $x\in \Omega$, $s\in [0,+\infty)$
\beq\label{ies1}
\left|\xi(s,t,x) \right|\leq K \nu(t) \left|a'(s) \right|r_0(s)
\eeq

\beq\label{ies2}
\left|\dfrac{\partial\xi}{\partial t}(s,t,x) \right|\leq 2K \theta \nu(t) \left|a'(s) \right|
\eeq

\beq\label{ies3}
\left|\dfrac{\partial\xi}{\partial s}(s,t,x) \right|\leq  K  \nu(t) \left[ 
\left |a''(s) \right|r_0(s)+\theta\left|a'(s) \right| \right] 
\eeq


\begin{align}\label{ies5}
\left|\dfrac{\partial^2\xi}{\partial t^2}(s,t,x) \right| & \leq 4\nu^2(t)
\left[K \theta +\left|g^{(3)}(0)\right| \right]\left|a'(s) \right| \nonumber\\
& +K\nu(t) \left|a'(s) \right|r_0(s) \left[\left|v_{xt}(x,t)\right| 
 + \left|v_{xt}(x,t-s)\right|\right]  
\end{align}
\end{enumerate}

The above derivatives may be considered in the classical sense, as they are defined for $s\neq t$. 
\end{lemma}

\begin{proof}
Observe that 

$$\dfrac{\partial \xi}{\partial t}=a'(s) g''\left(\overline{v}_x^t(s) \right)\left[v_x(t)-v_x(t-s) \right]   $$

$$\dfrac{\partial \xi}{\partial s}=a''(s) \left[ g'\left(\overline{v}_x^t(s) \right)-g'(0) \right]+a'(s)g''\left(\overline{v}_x^t(s) \right)v_x(t-s) $$ 


\begin{equation*}
\dfrac{\partial^2 \xi}{\partial t^2}= a'(s)g^{(3)}\left(\overline{v}_x^t(s) \right)\left[v_x(t)-v_x(t-s) \right]^2+a'(s)g''\left(\overline{v}_x^t(s) \right)\left[v_{xt}(t)-v_{xt}(t-s) \right]
\end{equation*}

Repeated use of part (i) of Lemma \ref{le} triggers the result.  

\end{proof}

For sake of clarity and - last but not least - reader's convenience, we restate Lemma's \ref{lies} content and then achieve its proof.

\begin{lemma}\label{lies1x}
Under the assumption that \eqref{pr6} is fulfilled,  one has:

\begin{align}\label{lies1x0}
\int_\Omega v_{tt}^2(x,t)\ud x -2g'(0) \displaystyle \mathop{\lim}_{h\to 0_+}\dfrac{1}{h^2} 
Q \left(\triangle_h v_{xt},t,a \right) \leq &  C \bigg\{ F + \sqrt{F}\sqrt{\mathcal{E}(t)} \nonumber\\
& + \left[\nu(t)+\nu^3(t) \right] \mathcal{E}(t) + \sqrt{V_0} \mathcal{E}(t) \bigg\}
\end{align}

\end{lemma}

\begin{proof}
Derivate \eqref{p4} w.r.t. $t$ and apply  $\triangle_h$ on the resulting equation.  One gets:

\beq\label{lies2}
\triangle_h v_{tt}=\int_0^{+\infty} a'(s) \triangle_h  \left(g\left(\overline{v}_x^t(s) \right) \right)_{xt} \ud s + \triangle_h f_t
\eeq

Multiply the above by $\triangle_h v_t$, integrate on $\Omega\times [0,t]$ to obtain 

\begin{align}\label{lies3}
& \dfrac{1}{2}\int_\Omega  \left[\triangle_h v_t(x,t) \right]^2\ud x- \dfrac{1}{2} 
\int_\Omega\left[\triangle_h v_t(x,0) \right]^2\ud x \nonumber\\ 
& =-\int_0^t \int_\Omega \int_0^{+\infty} a'(\tau) \triangle_h g\left(\overline{v}_x^s(x,\tau) 
\right)_{s} \triangle_h v_{xt}(x,s)\ud\tau \ud x \ud s \nonumber\\
& + \int_0^t \int_\Omega \triangle_h f_t (x,s) \triangle_h v_t(x,s) \ud x \ud s
\end{align}

Observing that 

$$g\left(\overline{v}_x^s(x,\tau) \right)_{s}=g'\left(\overline{v}_x^s(x,\tau) \right)\left[v_x(x,s)-v_x(x,s-\tau) \right]  $$

leads to

\beq\label{lies4}
-\int_0^t \int_\Omega \int_0^{+\infty} a'(\tau) \triangle_h g\left(\overline{v}_x^s(x,\tau) \right)_{s} \triangle_h v_{xt}(x,s)\ud\tau \ud x \ud s =I_1+I_2+I_3+I_4
\eeq

where:

\beq\label{lies5}
I_1=-\int_0^t \int_\Omega \int_0^{+\infty}a'(\tau)\triangle_h v_{xt}(x,s)\triangle_h  g'\left(\overline{v}_x^s(x,\tau) \right)\left[v_x(s+h)-v_x(s+h-\tau) \right] \ud\tau \ud x \ud s
\eeq

\begin{align}\label{lies6}
I_2=-\int_0^t \int_\Omega \int_0^{+\infty}a'(\tau)\triangle_h v_{xt}(x,s)\left[g'\left(\overline{v}_x^s(x,\tau) \right)-g'(0) \right] \triangle_h  v_x(x,s)\ud\tau \ud x \ud s
\end{align}

\begin{align}\label{lies7}
I_3=g'(0)\int_0^t \int_\Omega \int_0^{+\infty}a'(\tau)\triangle_h v_{xt}(x,s)\left[\triangle_h v_x(s-\tau)-\triangle_h v_x(s) \right] \ud\tau \ud x \ud s
\end{align}

\begin{align}\label{lies8}
I_4=\int_0^t \int_\Omega \int_0^{+\infty}a'(\tau)\triangle_h v_{xt}(x,s)\left[g'\left(\overline{v}_x^s(x,\tau) \right)-g'(0)\right]\triangle_h v_x(s-\tau) \ud\tau \ud x \ud s
\end{align}

Integrating by parts w.r.t. $s$ leads to $I_1=I_{11}+I_{12}$, where:

\begin{align}\label{lies90}
I_{11} & =-\int_\Omega \int_0^{+\infty} a'(\tau)\triangle_h v_{x}(x,t) \triangle_h  g'\left(\overline{v}_x^t(x,\tau) \right)
\left[v_x(x,t+h)-v_x(x,t+h-\tau) \right] \ud\tau \ud x \nonumber\\ 
& + \int_0^t \int_\Omega \int_0^{+\infty} a'(\tau)\triangle_h v_{x}(x,s)  \triangle_h \left[ g''\left( \overline{v}_x^s(x,\tau)\right)\left( v_x(x,s)-v_x(x,s-\tau) \right) \right] \nonumber\\
& \left[v_x(x,s+h)-v_x(x,s+h-\tau) \right]\ud\tau \ud x \ud s \nonumber\\
& +  \int_0^t \int_\Omega \int_0^{+\infty}a'(\tau)\triangle_h v_x(x,s) \triangle_h g'\left( \overline{v}_x^s(x,\tau)\right) \left[ v_{xt}(x,s+h)-v_{xt}(x,s+h-\tau)\right]\ud\tau \ud x \ud s 
\end{align}

and

\begin{align}\label{lies91}
I_{12} & =\int_\Omega \int_0^{+\infty}a'(\tau) \triangle_h v_x(0) \triangle_h  g'\left( \overline{v}_x^0(x,\tau)\right) \left[ v_{x}(x,h)-v_{x}(x,h-\tau)\right]\ud\tau \ud x \nonumber\\
& -\int_0^t \int_\Omega a'(s+h) \triangle_h g'\left( \int_0^s  v_x(x,\lambda)\ud\lambda\right) v'_0(x)\triangle_h v_{x}(x,s)\ud x\ud s
\end{align}

Observe that

\begin{align}\label{lies92}
& \int_\Omega \int_0^{+\infty}a'(\tau) \triangle_h v_x(0) \triangle_h  g'\left( \overline{v}_x^0(x,\tau)\right) \left[ v_{x}(x,h)-v_{x}(x,h-\tau)\right]\ud\tau \ud x \nonumber\\
& =  \int_\Omega \left[v_x(h)-v_x(0) \right]\int_0^h a'(\tau) \left[ g' \left(\int_{h-\tau}^h v_x(\lambda)\ud\lambda \right)-g'(0) \right]\left[v_x(h)-v_x(h-\tau) \right] \ud\tau\nonumber\\ 
& - \int_\Omega a(h) \left[v_x(h)-v_x(0) \right]\left[ g' \left(\int_{0}^h v_x(\lambda)\ud\lambda \right)-g'(0) \right] v_x(h)\ud x
\end{align}

By integrating the first term by parts w.r.t. $\tau$ one gets

\begin{align}\label{lies93}
I_{12} & =  \int_\Omega \left[v_x(h)-v_x(0) \right]^2 a(h) \left[ g' \left(\int_{0}^h v_x(\lambda)\ud\lambda \right)-g'(0) \right] \ud x \nonumber\\
& - \int_\Omega \left[v_x(h)-v_x(0) \right] \int_0^h a(\tau)  g'' \left(\int_{h-\tau}^h v_x(\lambda)\ud\lambda \right)  v_x(h-\tau) \left[v_x(h)-v_x(h-\tau) \right]\ud\tau \ud x \nonumber\\
& - \int_\Omega \left[v_x(h)-v_x(0) \right] \int_0^h a(\tau) \left[ g' \left(\int_{h-\tau}^h v_x(\lambda)\ud\lambda \right)-g'(0) \right]v_{xt}(h-\tau)\ud\tau \ud x \nonumber\\
& - \int_\Omega a(h) \left[v_x(h)-v_x(0) \right] \left[ g' \left(\int_{0}^h v_x(\lambda)\ud\lambda \right)-g'(0) \right]v_x(h)\ud x \nonumber\\
& - \int_0^t  \int_\Omega a'(s+h) \triangle_h g' \left(\int_{0}^s v_x(\lambda)\ud\lambda \right) v_0'(x) \triangle_h v_x(s) \ud x \ud s
\end{align}

Next, dividing the above by $h^2$, passing to the limit for $h\to 0_+$ and using the fact that $v$ and its derivatives up to order 2 belong to $\mathscr{C}^2\left([0,T); L^2(\Omega) \right) $ leads to

\beq\label{lies10}
\dfrac{1}{h^2}I_1\displaystyle\xrightarrow[h\to 0_+]{\,} J_1+J_{01}
\eeq

where

\begin{align}\label{lies11}
J_1= & - \int_\Omega \int_0^{+\infty} \partial_2 \xi(\tau,t,x)v_{xt}(x,t) \left[v_x(x,t)-v_x(x,t-\tau) \right]\ud \tau \ud x \nonumber\\
& + \int_0^t \int_\Omega \int_0^{+\infty} \partial_{22}\xi(\tau,s,x) v_{xt}(x,s) \left[v_x(x,s)-v_x(x,s-\tau) \right]\ud \tau \ud x \ud s \nonumber\\
& + \int_0^t \int_\Omega \int_0^{+\infty} \partial_{2}\xi(\tau,s,x) v_{xt}(x,s) \left[v_{xt}(x,s)-v_{xt}(x,s-\tau) \right]\ud \tau \ud x \ud s
\end{align}

and 

\beq\label{j01}
J_{01}= -\int_0^t \int_\Omega a'(s) v_{xt}(x,s) g''\left( \overline{v}_x^s(x,s)\right) v_x(x,s)v'_0(x)\ud x \ud s
\eeq

The term $I_2$ can be re-written as

\begin{align}\label{lies12}
I_2= & -\dfrac{1}{2} \int_0^t \int_\Omega \int_0^{+\infty} \xi(\tau,s,x) \dfrac{\partial}{\partial s}\left| \triangle_h v_x \right|^2(x,s)\ud \tau \ud x \ud s \nonumber\\
= & -\dfrac{1}{2} \int_\Omega \int_0^{+\infty}\xi(\tau,t,x)\left| \triangle_h v_x(x,t) \right|^2\ud \tau \ud x \nonumber\\
& + \dfrac{1}{2} \int_0^t \int_\Omega \int_0^{+\infty} \partial_2 \xi(\tau,s,x)\left| \triangle_h v_x(x,s) \right|^2\ud \tau \ud x \ud s 
\end{align}

Dividing by $h^2$ and passing to the limit for $h\to 0_+$ one obtains

\beq\label{lies13}
\dfrac{1}{h^2}I_2\displaystyle\xrightarrow[h\to 0_+]{\,}J_2
\eeq

where

\begin{align}\label{lies14}
J_2= & -\dfrac{1}{2} \int_\Omega \int_0^{+\infty}\xi(\tau,t,x)\left| v_{xt}(x,t) \right|^2\ud \tau \ud x \nonumber\\
& + \dfrac{1}{2} \int_0^t \int_\Omega \int_0^{+\infty} \partial_2 \xi(\tau,s,x)\left|  v_{xt}(x,s) \right|^2\ud\tau \ud x \ud s 
\end{align}

Next, $I_3=I_{31}+I_{32}+I_{33}$, where 

\beq\label{lies15}
I_{31}=g'(0)\int_0^t \int_\Omega \int_0^s a'(\tau)\triangle_h v_{xt}(x,s) \triangle_h v_{x}(x,s-\tau) \ud\tau \ud x \ud s 
\eeq

\beq\label{lies16}
I_{32}=g'(0) \int_0^t \int_\Omega \int_s^{s+h} a'(\tau)\triangle_h v_{xt}(x,s) v_{x}(x,s+h-\tau) \ud\tau \ud x \ud s 
\eeq

\beq\label{lies17}
I_{33}=g'(0) a(0) \int_0^t \int_\Omega  \triangle_h v_{xt}(x,s) \triangle_h v_{x}(x,s) \ud x \ud s 
\eeq

Upon integration by parts w.r.t. $\tau$ leads to

\begin{align}\label{lies162}
I_{31} & =g'(0) \int_0^t \int_\Omega a(s) \triangle_h v_{xt}(x,s)\triangle_h v_{x}(x,0)\ud x \ud s  \nonumber\\
& - g'(0) a(0)  \int_0^t \int_\Omega \triangle_h v_{xt}(x,s)\triangle_h v_{x}(x,s)\ud x \ud s  \nonumber\\
& + g'(0) Q\left(\triangle_h v_{xt},a,t \right) 
\end{align}

The above implies, upon simplification and integration by parts w.r.t. $s$, that

\begin{align}\label{lies163}
I_3 & =g'(0) Q\left(\triangle_h v_{xt},a,t \right)+ g'(0)\int_\Omega a(t) \triangle_h v_{x}(t)\triangle_h v_{x}(0) \ud x \nonumber\\
& - g'(0) a(0) \int_\Omega \left(\triangle_h v_{x}(0) \right)^2 \ud x - g'(0) \int_0^t \int_\Omega a'(s) \triangle_h v_{x}(s) \triangle_h v_{x}(0) \ud x \ud s \nonumber\\
& - g'(0)\int_0^t \int_\Omega \int_s^{s+h} a'(\tau)\triangle_h v_{t}(x,s) v_{xx}(x,s+h-\tau) \ud\tau \ud x \ud s
\end{align}
 
Divide the above by $h^2$ and taking the lower limit for $h\to0_+$ gives

\beq\label{lies164}
\mathop{\liminf}_{h\to0_+}  \dfrac{1}{h^2}I_3=g'(0) \mathop{\liminf}_{h\to0_+} \dfrac{1}{h^2} Q\left(\triangle_h v_{xt},a,t \right)+ J_3  
\eeq

where

\begin{align}\label{lies165}
J_3 & = g'(0) \bigg\{ a(t)\int_\Omega v_{xt}(x,t) v_{xt}(x,0)\ud x - a(0)\int_\Omega v_{xt}^2(x,0)\ud x  \nonumber\\
& - \int_0^t \int_\Omega a'(s)v_{xt}(x,s)v_{xt}(x,0) \ud x \ud s  - 
\int_0^t \int_\Omega a'(s) v_{tt}(x,s)v_0''(x)\ud x \ud s \bigg\}  
\end{align}

Next we end up with the same result as in \eqref{lies164} with $\left( \displaystyle \mathop{\liminf}_{h\to0_+}\right) $  being replaced by $\left( \displaystyle \mathop{\limsup}_{h\to0_+}\right) $.

Now we can write $I_4$ in the form:

\begin{align}\label{lies166} 
I_4 & = \int_0^t \int_\Omega \left[\int_0^s \xi(\tau,s)\triangle_h v_x (s-\tau)\ud\tau + \int_s^{s+h}\xi(\tau,s)v_x(s+h-\tau)\ud\tau  \right] \triangle_h v_{xt}(x,s) \ud x \ud s
\end{align}

An integration by parts w.r.t. $s$ gives

\beq\label{lies167}
I_4=I_{41}+I_{42}+I_{43}+I_{44}
\eeq

where

\begin{align}\label{lies168} 
I_{41} & = - \int_0^t \int_\Omega \int_0^s \left[ \partial_2 \xi(\tau,s)\triangle_h v_x (x,s-\tau)  + \xi(\tau,s)\triangle_h v_{xt} (x,s-\tau) \right] \ud\tau \triangle_h v_{x}(x,s) \ud x \ud s
\end{align}

\begin{align}\label{lies169}
I_{42} & = - \int_0^t \int_\Omega \int_s^{s+h} \left[ \partial_2 \xi(\tau,s) v_x (x,s+h-\tau)  + \xi(\tau,s) v_{xt} (x,s+h-\tau) \right] \ud\tau \triangle_h v_{x}(x,s) \ud x \ud s
\end{align}

\begin{align}\label{lies6n0}
I_{43} & = - \int_0^t \int_\Omega \left[\xi(s+h,s)-\xi(s,s) \right] v_0'(x) \triangle_h v_{x}(x,s) \ud x \ud s
\end{align}

and

\begin{align}\label{lies6n1}
I_{44} & = \left[\int_\Omega \int_0^{s+h} \xi(\tau,s)\triangle_h v_{x}(x,s-\tau)\triangle_h v_{x}(x,s)\ud\tau \ud x  \right]_{s=0}^{s=t} 
\end{align}

We now deal with the second term in $I_{41}$; we have:

\begin{align}\label{lies6n2}
- \int_0^s \xi(\tau,s)\triangle_h v_{xt}(x,s-\tau)\ud\tau & = \xi(s,s)\left[v_x(h)-v_x(0) \right]- \int_0^s \partial_1 \xi(\tau,s)\triangle_h v_{x}(x,s-\tau)\ud\tau 
\end{align}

fact that allows to get

\begin{align}\label{lies6n3}
I_{41} & = - \int_0^t \int_\Omega \int_0^s \left[ \partial_1 \xi(\tau,s)+\partial_2 \xi(\tau,s)\right] \triangle_h v_x (x,s-\tau)\triangle_h v_x (x,s)\ud\tau \ud x \ud s\nonumber\\
& + \int_0^t \int_\Omega \xi(s,s)\left[v_x(h)-v_x(0) \right]\triangle_h v_{x}(x,s) \ud x \ud s
\end{align}

Now we obtain

\begin{align}\label{lies6n4}
\dfrac{1}{h^2}I_4\xrightarrow[h\to 0_+]{\,}J_4+J_{04}
\end{align}

where

\begin{align}\label{lies6n5}
J_4= & - \int_0^t \int_\Omega \int_0^s \left[ \partial_1 \xi(\tau,s)+\partial_2 \xi(\tau,s)\right] v_{xt} (x,s-\tau) v_{xt} (x,s)\ud\tau \ud x \ud s  \nonumber\\
& + \int_0^t \int_\Omega \xi(\tau,t)v_{xt} (x,t-\tau) v_{xt} (x,t)\ud x \ud\tau  
\end{align}

and
 
\begin{align}\label{lies6n6}
J_{04} & =-\int_0^t \int_\Omega   \left[ \partial_1 \xi(s,s)+\partial_2 \xi(s,s)\right] v_0'(x) v_{xt} (x,s) \ud x \ud s  
\end{align}

Now, from \eqref{lies3}, \eqref{lies4},\eqref{lies10},\eqref{lies11},\eqref{lies13},\eqref{lies14},\eqref{lies164},\eqref{lies165},\eqref{lies6n4},\eqref{lies6n5}, we deduce that

\begin{align}\label{lies6n7}
& \dfrac{1}{2}\int_\Omega v_{tt}^2(x,t)\ud x- \dfrac{1}{2}\int_\Omega v_{tt}^2(x,0)\ud x = g'(0) \displaystyle\mathop{\lim}_{h\to0_+}\dfrac{1}{h^2}Q\left(\triangle_h v_{xt},a,t \right) \nonumber\\
& + \int_0^t \int_\Omega v_{tt}(x,s)f_{tt}(x,s)\ud x\ud s + J_1+J_2+J_3+J_4+J_{01}+J_{04}
\end{align}

with $J_1$-$J_4$ being given by \eqref{lies11}, \eqref{lies14}, \eqref{lies165} and 
\eqref{lies6n5}, respectively.  One now needs to appropriately bound the terms 
$J_1$-$J_4$, $J_{01}$ and $J_{04}$.  It may be easily seen, using 
Lemma \ref{ies}, that all terms $J_1$, $J_2$ and 
$J_4$ can be bounded by one of the following type of expressions:

\beq\label{bfj1}
c \nu^k (t) \int_\Omega \left|w_1(x,t) \right| \left|w_2(x,t) \right|\ud x 
\eeq

or 

\beq\label{bfj2}
c \nu^k (t) \int_0^t \int_\Omega \left|w_1(x,s) \right| \left|w_2(x,s) \right|\ud x\ud s
\eeq

or

\beq\label{bfj3}
c \nu^k (t) \int_0^t \int_\Omega \f(\tau) \left|w_1(x,t-\tau) \right| \left|w_2(x,t) \right|\ud x\ud \tau
\eeq

or

\beq\label{bfj4}
c \nu^k (t) \int_0^t \int_\Omega \int_0^s \f(\tau) \left|w_1(x,s-\tau) \right| \left|w_2(x,s) \right|\ud \tau\ud x\ud s
\eeq

where $\f\geq0$ is a given function in $L^1(\mathbb{R}_+)$ depending on $a$, $c>0$ is a constant, $w_1, w_2$ stand for either $v$ or one of its derivatives up to second order, and $k\in \{1,2,3\}$.  This is a consequence of assumption $(a_3)$. 

Terms like  \eqref{bfj1} and \eqref{bfj2} can easily be bounded by $c \nu^k(t)\mathcal{E}(t)$.  
Using 
Lemma \ref{adl2}, terms like \eqref{bfj3} and \eqref{bfj4} can also be easily bounded by 
$c \nu^k(t)\mathcal{E}(t)$.  We then obtain that there exists a constant $c>0$ s.t. 

\beq\label{bfj5}
J_1+J_2+J_4\leq c\left[\nu(t)+\nu^3(t) \right]\mathcal{E}(t)
\eeq

The estimates for $J_3$, $J_{01}$ and  $J_{04}$ are simpler to obtain since they contain initial data.  Using \eqref{eel9} we get $v_{xt}(x,0)=f_x(x,0)$.  It easily follows that

\beq\label{bfj6}
\left|J_3 \right|\leq \left|g'(0) \right| \left( \left|a(t) \right|+ \|a'\|_{L^1(\mathbb{R}_+)}\
\right) \left[ \left (\sqrt{F} + \sqrt{V_0} \right) \sqrt{\mathcal{E}(t)}+a(0)F \right]
\eeq 

\beq\label{bfj6n}
\left|J_{01} \right|\leq K \theta \|a'\|_{L^1(\mathbb{R}_+)} \|v_0\|_{H^2(\Omega)}\mathcal{E}(t) 
\eeq

and

\beq\label{bfj7}
\left|J_{04} \right| \leq 3K 
\left( \theta \|a'\|_{L^1(\mathbb{R}_+)} + \|a'' r_0\|_{L^1(\mathbb{R}_+)} \right)
\|v_0\|_{H^2(\Omega)}\nu(t)\sqrt{\mathcal{E}(t)}  
\eeq 

From \eqref{lies6n7}, \eqref{bfj5},\eqref{bfj6} and \eqref{bfj7}, the result stated in the Lemma now follows.

\end{proof}




\end{document}